\documentclass[12pt]{amsart}
\usepackage[utf8]{inputenc}
\usepackage[table]{xcolor}
\usepackage{mathtools,tikz-cd,amsmath,amssymb,graphicx}
\usepackage[shortlabels]{enumitem}
\usepackage[parfill]{parskip}

\usepackage{amsthm,amssymb,latexsym,epic,bbm,comment,mathrsfs,graphicx}
\usepackage[all,2cell]{xy}
\xyoption{2cell}

\allowdisplaybreaks

\newtheorem{theorem}{Theorem}
\newtheorem{conjecture}[theorem]{Conjecture}

\newtheorem{problem}[theorem]{Problem}
\newtheorem{lemma}[theorem]{Lemma}
\newtheorem{corollary}[theorem]{Corollary}

\newtheorem{proposition}[theorem]{Proposition}

\newtheorem{thmx}{Theorem}
\newtheorem{probx}[thmx]{Problem}
\theoremstyle{definition}

\newtheorem{remark}[theorem]{Remark}

\newcommand{\tto}{\twoheadrightarrow}

\newcommand{\Z}{\mathbb Z}
\newcommand{\inv}{^{-1}}

\newcommand{\cO}{\mathcal O}

\newcommand{\pd}{\operatorname{proj.dim}}
\newcommand{\wi}{w_0^{\mathfrak{p}}}

\usepackage{todonotes}

\begin{document}

\title[Some homological properties of category $\mathcal{O}$, VI]
{Some homological properties of category $\mathcal{O}$, VI}

\author[H.~Ko, V.~Mazorchuk and R.~Mr{\dj}en]
{Hankyung Ko, Volodymyr Mazorchuk and Rafael Mr{\dj}en}

\begin{abstract}
This paper explores various homological regularity phenomena 
(in the sense of Auslander) in category $\mathcal{O}$ and 
its several variations and generalizations. Additionally, we 
address the problem of determining projective dimension of
twisted and shuffled projective and tilting modules.
\end{abstract}

\maketitle

\setcounter{tocdepth}{1}
\tableofcontents

\section{Introduction, motivation and description of the results}\label{s1}

\subsection{Category $\mathcal{O}$}\label{s1.1}

Let $\mathfrak{g}$ be a semi-simple, finite dimensional Lie algebra over $\mathbb{C}$
with a fixed triangular decomposition
\begin{displaymath}
\mathfrak{g}=\mathfrak{n}_-\oplus \mathfrak{h}\oplus\mathfrak{n}_+. 
\end{displaymath}
Consider the Bernstein-Gelfand-Gelfand (BGG) {\em category $\mathcal{O}$} associated to
this decomposition. Category $\cO$ plays an important role in
modern representation theory and its applications. See e.g., \cite{BGS,Hu,So,Str} 
and references therein. Indecomposable blocks of $\mathcal{O}$
are described by finite dimensional algebras and possess a number of remarkable
symmetries. For example, they have simple preserving duality and exhibit both
Ringel self-duality and Koszul self-duality. See \cite{So,BGS,So2}.

Category $\mathcal{O}$ has a number of interesting sub- and quotient- categories such as the {\em parabolic category $\mathcal{O}$} associated with the choice of
a parabolic subalgebra $\mathfrak{p}$ of $\mathfrak{g}$ (see \cite{RC}) and 
the  {\em $\mathcal{S}$-subcategories in $\mathcal{O}$} associated  with $\mathfrak{p}$ (see \cite{FKM}). The latter categories are also known as the subcategories of 
{\em $\mathfrak{p}$-presentable modules}, see \cite{MS}, and can be
alternatively defined as certain Serre quotients of category $\mathcal{O}$.

\subsection{Auslander regular algebras}\label{s1.2}

A finite dimensional (associative) algebra $A$ is called {\em Auslander-Gorenstein},
see \cite{Iy,CIM}, provided that the (left) regular module ${}_AA$ 
admits a finite injective coresolution
\begin{displaymath}
0\to A\to Q_0\to Q_1\to\dots\to Q_k\to 0, 
\end{displaymath}
such that $\mathrm{proj.dim}(Q_i)\leq i$, for all $i=0,1,\dots,k$.
An Auslander-Gorenstein algebra of finite global dimension is called an
{\em Auslander regular} algebra.

Auslander regular algebras have a number of remarkable homological  properties,
see,  for example, \cite[Theorem~1.1]{Iy} and \cite[Theorem~2.1]{AR}.

We identify properties of algebras with that of their module categories, so,
in an appropriate case, we can say that $A$-mod is Auslander regular, etc.

\subsection{Motivation}\label{s1.3}

This paper originates from a question which the second author received
from Ren{\'e} Marczinzik in July 2020. The question was whether blocks of
category $\mathcal{O}$ are Auslander regular. It was motivated by 
the observations that the answer is positive in small ranks based on
computer calculations using the quiver and relation presentations
of blocks of category $\mathcal{O}$ from \cite{St1}.

\subsection{The main result}\label{s1.4}

The main result of the present paper is the following
statement which, in particular, answers positively 
and vastly generalizes the question posed 
by Ren{\'e} Marczinzik (see Theorem~\ref{thm3}, 
Corollary~\ref{cor6}, Theorem~\ref{thm9}, 
Corollary~\ref{cor10}, Theorem~\ref{thm12} and
Theorem~\ref{thm14}):

\begin{thmx}\label{thmA1}
All blocks of (parabolic) category $\mathcal{O}$ 
are Auslander regular. All blocks of 
$\mathcal{S}$-sub\-ca\-te\-go\-ri\-es in $\mathcal{O}$
are Auslander-Gorenstein.
\end{thmx}

The first two papers \cite{Ma3} and \cite{Ma4} of the 
``Some homological properties of category $\mathcal{O}$'' series 
were devoted to the study of projective dimension of structural modules
in category $\mathcal{O}$, with the main emphasis on the projective
dimension of indecomposable tilting and injective modules.
Our proof of Theorem~\ref{thmA1} is heavily based on these results.

\subsection{General setup for similar regularity phenomena}\label{s1.5}

We observe that the condition used to define Auslander-Gorenstein and
Auslander regular algebras makes perfect sense in the general setup of 
(generalized) tilting modules in the sense of Miyashita \cite{Mi}.
Let $A$ be a finite-dimensional algebra and $T$ an $A$-module.
Recall, that $T$ is called a {\em (generalized) tilting module}
provided that it has the following properties:
\begin{itemize}
\item $T$ has finite projective dimension;
\item $T$ is ext-self-orthogonal, that is, 
all extensions of positive degree from $T$ to $T$ vanish;
\item the module ${}_AA$ has a finite coresolution by modules
in $\mathrm{add}(T)$.
\end{itemize}
It is a standard fact that $\mathrm{proj.dim}(T)$ equals the length of 
a minimal coresolution of $A$ by modules from $\mathrm{add}(T)$.

Now, given $A$ and a 
(generalized) tilting $A$-module $T$, we say that
$A$ is $T$-regular provided that there is a coresolution
\begin{displaymath}
0\to A\to Q_0\to Q_1\to\dots\to Q_k\to 0, 
\end{displaymath}
such that $Q_i\in\mathrm{add}(T)$ and $\mathrm{proj.dim}(Q_i)\leq i$, 
for all $i=0,1,\dots,k$.

The notion of an Auslander-Gorenstein algebra corresponds to the situation
when the injective cogenerator is a (generalized) tilting module.

\subsection{Regularity phenomena for various generalized tilting
modules in category $\mathcal{O}$}\label{s1.6}

The bounded derived category of the principal block $\mathcal{O}_0$ of 
category $\mathcal{O}$ admits two different actions, by 
derived equivalences, of the 
braid group associated to $(W,S)$ where $W$ is the Weyl group of $\mathfrak g$ and $S$ the set of simple reflections.
These actions are given by the so-called {\em twisting functors}, 
see \cite{AS,KM}, and {\em shuffling functors}, see \cite{MS}.
These actions can be used to define the following
four classes of (generalized) tilting modules in $\mathcal{O}_0$:
\begin{itemize}
\item twisted projective modules;
\item twisted tilting modules;
\item shuffled projective modules;
\item shuffled tilting modules.
\end{itemize}

In Sections~\ref{s8} and \ref{s9} we explore the regularity phenomena in 
$\mathcal{O}_0$ with respect to these four families of 
(generalized) tilting modules. Each of these families contains
$|W|$ (generalized) tilting modules with some overlap between the families.

\begin{probx}\label{quesregularity}
For which of the above generalized tilting modules the category $\cO_0$ has the regularity property?
\end{probx}

Here is a summary of our results,
see Theorems~\ref{thm8.n1} and \ref{thm8.n21},
Propositions~\ref{prop9.n1} and \ref{prop9.n2}
and Examples in Subsections~\ref{s8.5} and \ref{s9.4}:

\begin{thmx}\label{thmA2}
\hspace{2mm}

\begin{enumerate}[$($a$)$]
\item\label{thmA2.1} The category $\mathcal{O}_0$ has the regularity
property with respect to both projective and tilting modules twisted
by the longest element in a parabolic subgroup of the Weyl group.
\item\label{thmA2.2} The category $\mathcal{O}_0$ has the regularity
property with respect to both projective and tilting modules shuffled
by a simple reflection.
\item\label{thmA2.3} There exist both twisted and shuffled
projective and tilting modules, with respect to which 
the category $\mathcal{O}_0$ does not have the regularity
property.
\end{enumerate}
\end{thmx}

\subsection{Projective dimension of twisted and shuffled projective
and tilting modules}\label{s1.7}

Theorem~\ref{thmA2} suggests that a complete answer to Problem~\ref{quesregularity} is non-trivial.
One important step here is the following problem.

\begin{probx}\label{prodim}
Determine the projective dimensions of twisted and shuffled projective and tilting modules in $\cO_0$. 
\end{probx}
We explore Problem \ref{prodim} in Section~\ref{s10}. 
Since twisted projective modules coincide with translated
Verma modules, while twisted tilting modules coincide with 
translated dual Verma modules,  
Problem \ref{prodim} provides a nice connection to the more recent
papers \cite{CM,KMM} in the ``Some homological properties of 
category $\mathcal{O}$'' series.
One of the main results of \cite{KMM}
determines projective dimension of translated simple modules in 
$\mathcal{O}$. In Section~\ref{s10} we propose conjectures
for projective dimension of twisted and shuffled projective
and tilting modules in the spirit of the results of \cite{KMM}
and prove a number of
partial results. All these conjectures and results are
formulated in terms of Kazhdan-Lusztig combinatorics, namely,
Lusztig's $\mathfrak{a}$-function from \cite{Lu1,Lu2}
and its various generalizations
studied in \cite{CM} and \cite{KMM}. The case of shuffled modules
seems at the moment to be significantly more difficult than the case
of twisted modules. The main reason for this is the fact that,
in contrast to twisting functors, shuffling functors do not commute
with projective functors.

\subsection*{Acknowledgments}

This research was partially supported by the Swedish Research Council, 
G{\"o}ran Gustafsson Stiftelse and Vergstiftelsen. The third author 
was also partially supported by the QuantiXLie Center of Excellence 
grant no. KK.01.1.1.01.0004 funded by the European Regional Development Fund.

We are especially indebted to Ren{\'e} Marczinzik for the question about Auslander 
regularity of $\mathcal{O}$, which started the research presented in this paper,
and also for his comments on the preliminary version of the manuscript.

\section{Auslander-Ringel regular quasi-hereditary algebras}\label{s2}

\subsection{Quasi-hereditary algebras}\label{s2.1}

Let $\Bbbk$ be an algebraically closed field and $A$ a finite dimensional
(associative) $\Bbbk$-algebra. Let $L_1,L_2,\dots,L_n$ be a complete and irredundant
list of isomorphism classes of simple $A$-modules. Note that, by fixing this list,
we have fixed a linear order on the isomorphism classes of simple $A$-modules, this will
be an essential part of the structure we are going to define now.

For $i\in\{1,2,\dots,n\}$, we denote by $P_i$ and $I_i$ the indecomposable
projective cover and injective envelope of $L_i$, respectively. Denote by
$\Delta_i$ the quotient of $P_i$ by the trace in $P_i$ of all 
$P_j$ with $j>i$. Denote by
$\nabla_i$ the submodule of $I_i$ defined as the intersection of 
the kernels of all homomorphisms from $I_i$ to $I_j$ with $j>i$. 
The modules $\Delta_i$ are called {\em standard} and the modules
$\nabla_i$ are called {\em costandard}.

Recall from \cite{CPS,DR}, that $A$ is said to be {\em quasi-hereditary}
provided that
\begin{itemize}
\item the endomorphism algebra of each $\Delta_i$ is $\Bbbk$;
\item the regular module ${}_AA$ has a filtration with standard subquotients.
\end{itemize}
According to \cite{Ri}, if  $A$ is quasi-hereditary, then, for each $i$,
there is a unique indecomposable module $T_i$, called a {\em tilting module},
which has both, a filtration with standard subquotients and a filtration with
costandard subquotients, and, additionally, such that $[T_i:L_i]\neq 0$
while $[T_i:L_j]=0$ for $j>i$. The module 
$\displaystyle T=\bigoplus_{i=1}^n T_i$ is called the
{\em characteristic tilting module} and (the opposite of) 
its endomorphism algebra is called the {\em Ringel dual} of $A$.

For each $M\in A$-mod, there is a unique minimal finite complex
$\mathcal{T}_\bullet(M)$ of tilting modules which is isomorphic to $M$
in the bounded derived category of $A$. We will denote by $\mathbf{r}(M)$
the maximal non-negative $i$ such that $\mathcal{T}_i(M)\neq 0$ and by 
$\mathbf{l}(M)$ the maximal non-negative $i$ such that $\mathcal{T}_{-i}(M)\neq 0$.
Note that $\mathbf{l}(M)=0$ if and only if $M$ has a filtration with standard
subquotients and $\mathbf{r}(M)=0$ if and only if $M$ has a filtration with costandard
subquotients.
We refer to \cite{MO2} for further details.

\subsection{Auslander-Ringel regular algebras}\label{s2.2}

We say that a quasi-hereditary algebra $A$ is {\em Auslander-Ringel regular} 
provided that there is a coresolution
\begin{displaymath}
0\to A\to Q_0\to Q_1\to\dots\to Q_k\to 0, 
\end{displaymath}
such that each $Q_i\in\mathrm{add}(T)$ and $\mathrm{proj.dim}(Q_i)\leq i$, 
for all $i=0,1,\dots,k$. 
Note that, being quasi-hereditary, $A$ has  finite global dimension (see \cite{CPS,DR}) and that the characteristic tilting module is a
(generalized) tilting module. Thus, Auslander-Ringel regularity corresponds to
$T$-regularity in the terminology of Subsection~\ref{s1.5}.

In Section~\ref{s3}, we will see that blocks of (parabolic) 
BGG category $\mathcal{O}$ are Auslander-Ringel regular.

\section{Regularity phenomena in category $\mathcal{O}$}\label{s3}

\subsection{Category $\mathcal{O}$}\label{s3.1}

We refer the reader to \cite{Hu} for details and generalities about
category $\mathcal{O}$.

We denote by $\mathcal{O}_0$ the principal block of $\mathcal{O}$, that is, the
indecomposable direct summand of $\cO$ containing the trivial $\mathfrak{g}$-module.
The simple modules in $\mathcal{O}_0$ are simple highest weight
modules, and their isomorphism classes are naturally indexed by elements of the Weyl group $W$. 
For $w\in W$, we denote by $L_w$ the simple highest weight module
in $\mathcal{O}_0$ with highest weight $w\cdot 0$, where $0$ is the zero
element in $\mathfrak{h}^*$ and $\cdot$ is the dot action of $W$.

We denote by $P_w$ and $I_w$ the indecomposable projective cover and 
injective envelope of $L_w$ in $\mathcal{O}_0$, respectively. 
Let $A$ be a basic, finite dimensional, associative algebra such that
$\mathcal{O}_0$ is equivalent to $A$-mod. It is well-known that $A$
is quasi-hereditary with respect to any linear order which extends the
dominance order on weights. The latter is given by $\lambda\leq \mu$ if and only if
$\mu-\lambda$ is a linear combination of positive roots with 
non-negative integer coefficients.

By \cite{So}, the algebra $A$ admits a Koszul
$\mathbb{Z}$-grading.  We denote by ${}^{\mathbb{Z}}\mathcal{O}_0$
the category of $\mathbb{Z}$-graded finite-dimensional 
$A$-modules. We denote by
$\langle 1\rangle$ the shift of grading which maps degree $0$
to degree $-1$. We fix standard graded lifts of structural modules so that
\begin{itemize}
\item $L_w$ is concentrated in degree zero;
\item the top of $P_w$ is concentrated in degree zero;
\item the socle of $I_w$ is concentrated in degree zero;
\item the top of $\Delta_w$ is concentrated in degree zero;
\item the socle of $\nabla_w$ is concentrated in degree zero;
\item the canonical map $\Delta_w\hookrightarrow T_w$ is 
homogeneous of degree zero.
\end{itemize}

For $w\in W$, we denote by $\theta_w$ the indecomposable projective
endofunctor of $\mathcal{O}_0$, see \cite{BG}, uniquely defined by
the property $\theta_w P_e\cong P_w$. By \cite{St2}, $\theta_w$
admits a natural graded lift normalized by the same condition.

We denote by $\geq_{\mathtt{L}}$, $\geq_{\mathtt{R}}$
and $\geq_{\mathtt{J}}$ the Kazhdan-Luszitg left, right and two-sided
orders, respectively.

\subsection{$\mathcal{O}_0$ is Auslander-Ringel regular}\label{s3.2}

\begin{theorem}\label{thm1}
The category  $\mathcal{O}_0$ is Auslander-Ringel regular.
\end{theorem}

\begin{proof}
Consider the category $\mathcal{LT}(\mathcal{O}_0)$ of linear complexes of
tilting modules in $\mathcal{O}_0$, see \cite{Ma,MO}. The algebra $A$ 
is a balanced quasi-hereditary algebra in the sense of \cite{Ma2} and hence
$\mathcal{LT}(\mathcal{O}_0)$ contains the tilting coresolution 
$\mathcal{T}_\bullet(P_e)$ of the dominant standard module
$\Delta_e=P_e$. 

Due to the Ringel-Koszul self-duality of $\mathcal{O}_0$, 
the category $\mathcal{LT}(\mathcal{O}_0)$ is equivalent to 
${}^{\mathbb{Z}}\mathcal{O}_0$. This implies that the multiplicity
of $T_w\langle i\rangle$ as a summand of $\mathcal{T}_i(P_e)$ coincides with the
composition multiplicity of $L_{w_0w\inv w_0}\langle -i\rangle$ in $\Delta_e$. The latter 
is given by Kazhdan-Lusztig combinatorics for $(W,S)$. In particular, it is non-zero only
if 
\begin{displaymath}
\mathbf{a}(w)=\mathbf{a}(w_0w\inv w_0)\leq i\leq\ell(w_0w\inv w_0)=\ell(w), 
\end{displaymath}
where $\ell(w)$ is the length of 
$w$ and $\mathbf{a}$ is Lusztig's $\mathbf{a}$-function from \cite{Lu1,Lu2}.

Consequently, $T_w$ can appear (up to shift of grading) only in homological positions
$i$ such that $\mathbf{a}(w)\leq i\leq\ell(w)$. Taking into account that
$\mathrm{proj.dim.}(T_w)=\mathbf{a}(w)$ by \cite{Ma3,Ma4}, it follows that 
$\mathcal{T}_i(P_e)$ has projective dimension at most $i$.

For $x\in W$, applying $\theta_x$ to $\mathcal{T}_\bullet(P_e)$ gives a
tilting coresolution of $P_x$ (not necessarily minimal or linear). 
Since $\theta_x$ is exact and sends projectives to projectives, it cannot
increase the projective dimension. This means that
\begin{displaymath}
\mathrm{proj.dim.}(\theta_x\mathcal{T}_i(P_e))\leq
\mathrm{proj.dim.}(\mathcal{T}_i(P_e))\leq i.
\end{displaymath}
The claim of the theorem follows.
\end{proof}

\begin{corollary}\label{cor2}
{\hspace{1mm}}

\begin{enumerate}[$($i$)$]
\item\label{cor2.1} Let $\mathcal{P}_\bullet(T)$ be a minimal projective resolution
of $T$. Then $\mathbf{r}(\mathcal{P}_{-i}(T))\leq i$, for all $i\geq 0$.
\item\label{cor2.2} Let $\mathcal{T}_\bullet(I)$ be a minimal tilting resolution
of the basic injective cogenerator $I$. Then we have $\mathrm{inj.dim.}(\mathcal{T}_{-i}(I))\leq i$, 
for all $i\geq 0$.
\item\label{cor2.3} Let $\mathcal{I}_\bullet(T)$ be a minimal injective coresolution
of $T$. Then $\mathbf{l}(\mathcal{I}_i(T))\leq i$, for all $i\geq 0$.
\end{enumerate}
\end{corollary}

\begin{proof}
Claim~\ref{cor2.2} is obtained from Theorem~\ref{thm1} using the simple
preserving duality on $\mathcal{O}$. Since $\mathcal{O}_0$ is Ringel self-dual,
Claim ~\ref{cor2.1} is the Ringel dual of Theorem~\ref{thm1} and, finally,
Claim ~\ref{cor2.3} is the Ringel dual of Claim~\ref{cor2.2}.
\end{proof}

\subsection{$\mathcal{O}_0$ is Auslander regular}\label{s3.3}

\begin{theorem}\label{thm3}
The category $\mathcal{O}_0$ is Auslander regular. 
\end{theorem}

We will need the following auxiliary statement.

\begin{lemma}\label{lem4}
For $w\in W$, let $\mathcal{I}_\bullet(T_{w_0w})$ be a minimal injective coresolution
of $T_{w_0w}$. Then we have:
\begin{enumerate}[$($i$)$]
\item\label{lem4.1} The maximal value of $i$ such that 
$\mathcal{I}_i(T_{w_0w})\neq 0$ equals $\mathbf{a}(w_0w)$.
\item\label{lem4.2} Each indecomposable direct summand of $\mathcal{I}_\bullet(T_{w_0w})$ 
is isomorphic, up to a graded shift, to $I_x$, for some $x\geq_{\mathtt{J}}w$.
\item\label{lem4.3} If $\mathcal{I}_i(T_{w_0w})$ has a direct summand
isomorphic, up to a graded shift, to $I_x$, for some $x\in W$, 
then $i\geq \mathbf{a}(w_0x)$.
\end{enumerate}
\end{lemma}

\begin{proof}
Using the simple preserving duality, Claim~\ref{lem4.1} is one of the main results of 
\cite{Ma3,Ma4}. 

Since $T_{w_0w}\cong \theta_w T_{w_0}$, to prove  Claim~\ref{lem4.2}, it is enough to prove the same statement for  $\theta_w\mathcal{I}_\bullet(T_{w_0})$. 
But we have
\begin{displaymath}
\theta_w I_y= \theta_w \theta_y I_e=
\bigoplus_{z\in W}\theta_z^{\oplus m_{w,y}^z}I_e=
\bigoplus_{z\in W} I_z^{\oplus m_{w,y}^z}
\end{displaymath}
and  $m_{w,y}^z\neq 0$ only if $z\geq_{\mathtt{J}}w$.

Let us prove Claim~\ref{lem4.3}. We start with the case $w=e$.
Due to Koszulity
of $\mathcal{O}_0$, the minimal injective coresolution $\mathcal{I}_\bullet(T_{w_0})$ 
of the antidominant tilting$=$simple module $T_{w_0}=L_{w_0}$ is linear and
hence is an object in the category $\mathcal{LI}(\mathcal{O}_0)$ of linear complexes
of injective modules and is isomorphic to the dominant standard$=$projective
object in this category. 

Due to the Koszul self-duality of $\mathcal{O}_0$, 
the category $\mathcal{LI}(\mathcal{O}_0)$ is equivalent to 
${}^{\mathbb{Z}}\mathcal{O}_0$. This implies that the multiplicity
of $I_x\langle i\rangle$ as a summand of $\mathcal{I}_i(T_{w_0})$ coincides with the
composition multiplicity of $L_{w_0x^{-1}}\langle -i\rangle$ in $\Delta_e$. The latter 
is given by Kazhdan-Lusztig combinatorics. In particular, it is non-zero only
if $\mathbf{a}(w_0x\inv)\leq i\leq\ell(w_0x\inv)$. 
Consequently, the module 
$I_x$ can appear (up to shift of grading) only in homological positions
$i$ such that $\mathbf{a}(w_0x)=\mathbf{a}(w_0x\inv)\leq i\leq\ell(w_0x\inv)$. 
This proves Claim~\ref{lem4.3} in the case $w=e$. The general case
is obtained from this one applying $\theta_w$.
\end{proof}

\begin{proof}[Proof of Theorem~\ref{thm3}]
Take the minimal tilting coresolution $\mathcal{T}_\bullet(P_e)$ of $P_e$
considered in the proof of Theorem~\ref{thm1}. We can take a minimal injective
coresolution of each $T_x$, up to grading shift, appearing in $\mathcal{T}_\bullet(P_e)$ 
and glue these into an injective coresolution $\mathcal{I}_\bullet$
of $P_e$. Applying $\theta_w$ to $\mathcal{I}_\bullet$, gives an
injective coresolution of $P_w$ without increasing the projective dimensions
of homological positions. 
By \cite{Ma3,Ma4}, the projective dimension of $I_x$ is $2\mathbf{a}(w_0x)$. 
It is thus
enough to show that any graded shift of $I_x$ appearing in $\mathcal{I}_\bullet$ appears only in homological positions $i$ such that 
$i\geq 2\mathbf{a}(w_0x)$.

By Lemma~\ref{lem4}, $I_x$ can only appear in homological position at least
$\mathbf{a}(w_0x)$ when coresolving $T_y$. Furthermore, again by
Lemma~\ref{lem4}, $I_x$ can only appear in coresolutions of $T_{w_0y}$, where
$x\geq_{\mathtt{J}}y$. By Theorem~\ref{thm1}, such $T_{w_0y}$ appears in $\mathcal T_\bullet(P_e)$ in homological
positions at least $\pd T_{w_0y}= \mathbf{a}(w_0y)$.
Adding these two estimates together, we obtain that
$I_x$ appears in $\mathcal{I}_\bullet$ in homological positions at least $\mathbf{a}(w_0y)+\mathbf{a}(w_0x)\geq 2\mathbf{a}(w_0x)$.
This completes the proof.
\end{proof}

\subsection{Singular blocks}\label{s3.4}

\begin{corollary}\label{cor6}
All blocks of $\mathcal{O}$ are  Auslander-Ringel regular, Auslander regular,
and have the properties 
described in Corollary~\ref{cor2}. 
\end{corollary}

\begin{proof}
Due to Soergel's combinatorial description of blocks of $\mathcal{O}$
from \cite{So}, each block of category $\mathcal{O}$ is equivalent to an
integral block of $\mathcal{O}$ (possibly for a different Lie algebra).
Therefore we may restrict our attention to integral blocks. 

Each regular integral block is equivalent to $\mathcal{O}_0$. Each singular
integral block is obtained from a regular integral block using translation
to the corresponding wall. These translation functors are exact, send 
projectives to projectives, injectives to injectives and tiltings to tiltings
and do not increase projective dimension, injective dimension, nor
the values of $\mathbf{l}$ and $\mathbf{r}$. Therefore the claim  follows from Theorems~\ref{thm1} and \ref{thm3} and
Corollary~\ref{cor2} applying these translation functors.
\end{proof}

\subsection{$\mathfrak{sl}_3$-example}\label{s3.5}

For the Lie algebra $\mathfrak{sl}_3$, we have
$W=\{e,s,t,st,ts,w_0=sts=tst\}$. The projective dimensions
of the indecomposable tilting and injective modules in $\mathcal{O}_0$ 
are given by:
\begin{displaymath}
\begin{array}{c||c|c|c|c|c|c}
w&e&s&t&st&ts&w_0\\ \hline
\mathrm{proj.dim}(T_w)&0&1&1&1&1&3
\end{array}\qquad
\begin{array}{c||c|c|c|c|c|c}
w&e&s&t&st&ts&w_0\\ \hline
\mathrm{proj.dim}(I_w)&6&2&2&2&2&0
\end{array}
\end{displaymath}
The minimal (ungraded) tilting coresolutions of the
{\color{blue}indecomposable projectives} in $\mathcal{O}_0$ are:
\begin{displaymath}
0\to {\color{blue}P_{e}}\to T_e\to T_s\oplus T_t
\to T_{st}\oplus T_{ts}\to T_{w_0}\to 0,
\end{displaymath}
\begin{displaymath}
0\to {\color{blue}P_{s}}\to T_e\to  T_t \to 0,
\end{displaymath}
\begin{displaymath}
0\to {\color{blue}P_{t}}\to T_e\to  T_s \to 0,
\end{displaymath}
\begin{displaymath}
0\to {\color{blue}P_{st}}\to T_e\to  T_{ts} \to 0,
\end{displaymath}
\begin{displaymath}
0\to {\color{blue}P_{ts}}\to T_e\to  T_{st} \to 0,
\end{displaymath}
\begin{displaymath}
0\to {\color{blue}P_{w_0}}\to T_{e}\to 0,
\end{displaymath}
The minimal (ungraded) injective coresolutions of the
{\color{blue}indecomposable projectives} in $\mathcal{O}_0$ are:
$$
0\to {\color{blue}P_{e}}\to I_{w_0}
\to I_{w_0}^{\oplus 2}
\to I_t \oplus I_s\oplus I_{w_0}^{\oplus 2}
\to I_{ts}\oplus I_{st} \oplus I_{w_0}
\to I_{st}\oplus I_{ts}\to I_s\oplus I_t\to I_e\to 0,
$$
\begin{displaymath}
0\to {\color{blue}P_{s}}\to I_{w_0}\to I_{w_0}\to  I_s \to 0,
\end{displaymath}
\begin{displaymath}
0\to {\color{blue}P_{t}}\to I_{w_0}\to I_{w_0}\to  I_t \to 0,
\end{displaymath}
\begin{displaymath}
0\to {\color{blue}P_{st}}\to I_{w_0}\to I_{w_0}\to  I_{st} \to 0,
\end{displaymath}
\begin{displaymath}
0\to {\color{blue}P_{ts}}\to I_{w_0}\to I_{w_0}\to  I_{ts} \to 0,
\end{displaymath}
\begin{displaymath}
0\to {\color{blue}P_{w_0}}\to I_{w_0}\to 0,
\end{displaymath}

\section{Regularity phenomena in parabolic category $\mathcal{O}^{\mathfrak{p}}$}\label{s4}

\subsection{Parabolic category $\mathcal{O}^{\mathfrak{p}}$}\label{s4.1}

Fix a parabolic subalgebra $\mathfrak{p}$ of $\mathfrak{g}$ containing
$\mathfrak{h}\oplus\mathfrak{n}_+$. Denote by $\mathcal{O}^{\mathfrak{p}}$
the full subcategory of $\mathcal{O}$ consisting of all objects the action of
$U(\mathfrak{p})$ on which is locally finite, see \cite{RC}.
Then $\mathcal{O}^{\mathfrak{p}}$ is the Serre subcategory of $\mathcal{O}$
generated by all simple modules whose highest weights are (dot-)dominant (by which we mean it is the largest weight in its orbit under the dot action) 
and integral with respect to the Levi factor of $\mathfrak{p}$.

Similarly to Subsection~\ref{s3.4}, we can start with the integral regular
situation. Let $W_{\mathfrak{p}}$ denote the Weyl group of the Levi factor of 
$\mathfrak{p}$ which we view as a parabolic subgroup of $W$.
We denote by $w_0^{\mathfrak{p}}$ the longest element in $W_{\mathfrak{p}}$.
The principal block $\mathcal{O}^{\mathfrak{p}}_0$ is the Serre subcategory of
$\mathcal{O}^{\mathfrak{p}}_0$ generated by $L_w$, where $w$ belongs to the set
$\mathrm{short}({}_{W_{\mathfrak{p}}}\hspace{-2mm}\setminus\hspace{-1mm} W)$ 
of shortest coset representatives
for cosets in ${}_{W_{\mathfrak{p}}}\hspace{-2mm}\setminus\hspace{-1mm} W$.

\subsection{$\mathcal{O}^{\mathfrak{p}}_0$ is Auslander-Ringel regular}\label{s4.2}

\begin{theorem}\label{thm7}
The category  $\mathcal{O}^{\mathfrak{p}}_0$ is Auslander-Ringel regular.
\end{theorem}

\begin{proof}
The proof is similar to the proof of Theorem~\ref{thm1}, so we only emphasize the
differences. 
By \cite{BGS}, the Koszul dual of $\mathcal{O}^{\mathfrak{p}}_0$
is 
the singular integral block $\mathcal{O}_{\lambda}$ of $\mathcal{O}$ 
where $\lambda$ is chosen such that the dot-stabilizer of $\lambda$ equals $W_{\mathfrak{p'}}$ where $\mathfrak{p'}$ is the $w_0$-conjugate of $\mathfrak{p}$.
By \cite{So2}, the block $\cO_\lambda$ is Ringel self-dual, and by \cite{Ma2}, the Ringel duality and the Koszul duality commute.
Therefore, the
category of linear complexes of tilting modules in $\mathcal{O}^{\mathfrak{p}}_0$
is equivalent to ${}^{\mathbb{Z}}\mathcal{O}_{\lambda}$.

By \cite{Ma}, the tilting coresolution of the dominant projective ($=$standard)
module in $\mathcal{O}^{\mathfrak{p}}_0$ is $\Delta(\lambda)$, the dominant standard object in
${}^{\mathbb{Z}}\mathcal{O}_{\lambda}$. Denoting by $T_0^\lambda:{}^\Z\cO_0\to {}^\Z\cO_\lambda$ the graded translation functor to the $\lambda$-wall, we have $\Delta(\lambda)\cong T_0^\lambda \Delta_e\langle \ell(w_0^{\mathfrak{p'}}) \rangle $.
This means that the degree $i$ component of $\Delta(\lambda)$ consists of $T_0^\lambda L_u$ where $L_u$ belongs to the degree  $i+\ell(w_0^{\mathfrak{p}})=i+\ell(w_0^{\mathfrak{p'}})$ component of $\Delta_e$ and such that 
$u\in \mathrm{long}( W /_{W_{\mathfrak{p'}}})=w_0(\mathrm{long}( {}_{W_{\mathfrak{p}}} \hspace{-2mm}\setminus\hspace{-1mm} W))\inv w_0$.
It follows that the $i$-th component in the tilting coresolution contains only $T_x^{\mathfrak p}$ where $x\in \mathrm{short}( {}_{W_{\mathfrak{p}}} \hspace{-2mm}\setminus\hspace{-1mm} W)$ is such that
$\mathbf{a}(w_0(\wi x)\inv w_0)\geq i+\ell(\wi)=i+\mathbf{a}(\wi)$.

It remains to check from \cite[Table~2]{CM} that 
\[\pd T_x^{\mathfrak p} = \mathbf{a}(\wi x)-\mathbf{a}(\wi)= \mathbf{a}(w_0(\wi x )\inv w_0)-\mathbf{a}(\wi)\] and compare with the condition in the previous paragraph. 
This proves the regularity property for the tilting coresolution of the dominant projective. 

The regularity property for other projective modules in $\cO^{\mathfrak p}_0$ is obtained by applying projective functors exactly as in Theorem~\ref{thm1}.
\end{proof}

Let $P^{\mathfrak{p}}$ denote a projective generator,
$I^{\mathfrak{p}}$ an injective cogenerator,
and $T^{\mathfrak{p}}$ the characteristic tilting module in $\mathcal{O}^{\mathfrak{p}}_0$.
Similarly to Corollary~\ref{cor2} (using that $\cO^{\mathfrak p}_0$ is equivalent to its Ringel dual $\cO^{\mathfrak p'}_0$), we have:

\begin{corollary}\label{cor8}
{\hspace{1mm}}

\begin{enumerate}[$($i$)$]
\item\label{cor8.1} Let $\mathcal{P}_\bullet(T^{\mathfrak{p}})$ be a minimal 
projective resolution of $T^{\mathfrak{p}}$ in $\mathcal{O}^{\mathfrak{p}}_0$. 
Then $\mathbf{r}(\mathcal{P}_{-i}(T^{\mathfrak{p}}))\leq i$, for all $i\geq 0$.
\item\label{cor8.2} Let $\mathcal{T}_\bullet(I^{\mathfrak{p}})$ 
be a minimal tilting resolution of the basic injective cogenerator 
$I^{\mathfrak{p}}$ in $\mathcal{O}^{\mathfrak{p}}_0$. 
Then $\mathrm{inj.dim.}(\mathcal{T}_{-i}(I^{\mathfrak{p}}))\leq i$,  for all $i\geq 0$.
\item\label{cor8.3} Let $\mathcal{I}_\bullet(T^{\mathfrak{p}})$ be a minimal 
injective coresolution of $T^{\mathfrak{p}}$ in $\mathcal{O}^{\mathfrak{p}}_0$. Then 
$\mathbf{l}(\mathcal{I}_i(T^{\mathfrak{p}}))\leq i$, for all $i\geq 0$.
\end{enumerate}
\end{corollary}

\subsection{$\mathcal{O}_0^{\mathfrak{p}}$ is Auslander regular}\label{s4.3}

\begin{theorem}\label{thm9}
The category $\mathcal{O}^{\mathfrak{p}}_0$ is Auslander regular. 
\end{theorem}

\begin{proof}
Mutatis mutandis the proof of Theorem~\ref{thm3}.
Again, one could emphasize the 
$2\mathfrak{a}(w_0^{\mathfrak{p}})=2\ell(w_0^{\mathfrak{p}})$ shift 
for the projective dimension of injective modules in $\mathcal{O}^{\mathfrak{p}}_0$ in
\cite[Table~2]{CM}.
\end{proof}

\subsection{Singular blocks}\label{s4.4}

\begin{corollary}\label{cor10}
All blocks of $\mathcal{O}^{\mathfrak{p}}$ are both Auslander-Ringel 
regular and Auslander regular and have the properties described in Corollary~\ref{cor8}. 
\end{corollary}

\begin{proof}
Mutatis mutandis the proof of Corollary~\ref{cor6}.
\end{proof}

\subsection{$\mathfrak{sl}_3$-example}\label{s4.5}

For the Lie algebra $\mathfrak{sl}_3$, we have
$W=\{e,s,t,st,ts,w_0=sts=tst\}$. Assume that $W_\mathfrak{p}=\{e,s\}$,
then  $\mathrm{short}({}_{W_{\mathfrak{p}}}\hspace{-2mm}\setminus
\hspace{-1mm} W)=\{e,t,ts\}$.
The projective dimensions
of the indecomposable tilting and projective modules in 
$\mathcal{O}_0^{\mathfrak{p}}$ 
are given by:
\begin{displaymath}
\begin{array}{c||c|c|c}
w:&e&t&ts\\ \hline
\mathrm{proj.dim}(T_w^{\mathfrak{p}}):&0&0&2
\end{array}\qquad
\begin{array}{c||c|c|c}
w:&e&t&ts\\ \hline
\mathrm{proj.dim}(I_w^{\mathfrak{p}}):&4&0&0
\end{array}
\end{displaymath}
The minimal (ungraded) tilting coresolutions of the
{\color{blue}indecomposable projectives} in $\mathcal{O}_0^{\mathfrak{p}}$ are:
\begin{displaymath}
0\to {\color{blue}P_{e}^{\mathfrak{p}}}\to 
T_e^{\mathfrak{p}}\to T_t^{\mathfrak{p}}
\to T_{ts}^{\mathfrak{p}}\to 0,
\end{displaymath}
\begin{displaymath}
0\to {\color{blue}P_{t}^{\mathfrak{p}}}\to T_{e}^{\mathfrak{p}}\to 0,
\end{displaymath}
\begin{displaymath}
0\to {\color{blue}P_{ts}^{\mathfrak{p}}}\to T_{t}^{\mathfrak{p}}\to 0,
\end{displaymath}
The minimal (ungraded) injective coresolutions of the
{\color{blue}indecomposable projectives} in $\mathcal{O}_0^{\mathfrak{p}}$ are:
\begin{displaymath}
0\to {\color{blue}P_{e}^{\mathfrak{p}}}\to 
I_t^{\mathfrak{p}} \to I_s^{\mathfrak{p}}\to I_s^{\mathfrak{p}}\to 
I_{t}^{\mathfrak{p}}\to I_e^{\mathfrak{p}}\to 0,
\end{displaymath}
\begin{displaymath}
0\to {\color{blue}P_{t}^{\mathfrak{p}}}\to I_t^{\mathfrak{p}} \to 0,
\end{displaymath}
\begin{displaymath}
0\to {\color{blue}P_{ts}^{\mathfrak{p}}}\to I_{ts}^{\mathfrak{p}} \to 0,
\end{displaymath}

\section{Auslander-Ringel-Gorenstein strongly standardly stratified algebras}\label{s5}

\subsection{Strongly standardly stratified algebras}\label{s5.1}

In this section we return to the general setup of Subsection~\ref{s2.1}.

For $i\in\{1,2,\dots,n\}$, we denote by 
$\overline{\Delta}_i$ the maximal quotient of $\Delta_i$ 
satisfying $[\overline{\Delta}_i:L_i]=1$. Denote by
$\overline{\nabla}_i$ the maximal submodule of $\nabla_i$ 
satisfying $[\overline{\nabla}_i:L_i]=1$.
The modules $\overline{\Delta}_i$ are called {\em proper standard} and the modules
$\overline{\nabla}_i$ are called {\em proper costandard}.

Recall that $A$ is said to be {\em standardly stratified} provided that the regular module ${}_AA$ has a filtration with standard subquotients and {\em strongly strandardly stratified} (see \cite{Fr}) if, further, each standard module has a filtration with proper standard subquotients.

If $A$ is a strongly standardly stratified algebra, then, by \cite{AHLU}, for each $i$,
there is a unique indecomposable module $T_i$, called a {\em tilting module},
which has both a filtration with standard subquotients and a filtration with
proper costandard subquotients, and, additionally, such that $[T_i:L_i]\neq 0$
while $[T_i:L_j]=0$, for $j>i$. The module 
$\displaystyle T=\bigoplus_{i=1}^n T_i$ is called the
{\em characteristic tilting module} and (the opposite of) 
its endomorphism algebra is called the {\em Ringel dual} of $A$.
For each $M\in A$-mod, there is a unique minimal bounded from the right complex
$\mathcal{T}_\bullet(M)$ of tilting modules which is isomorphic to $M$
in the bounded derived category of $A$. We will denote by $\mathbf{r}(M)$
the maximal non-negative $i$ such that $\mathcal{T}_i(M)\neq 0$.
Note that $\mathbf{r}(M)=0$ if and only if $M$ has a filtration with proper 
costandard subquotients.

\subsection{Auslander-Ringel-Gorenstein  algebras}\label{s5.2}

Let $A$ be strongly standardly stratified.
Then an $A$-module having a filtration
with standard subquotients has a (finite) coresolution by modules in
$\mathrm{add}(T)$. It is also well-known that $T$ has finite projective dimension (see \cite{Fr, AHLU}).
We will say that $A$ is {\em Auslander-Ringel-Gorenstein} 
provided that there is a coresolution
\begin{displaymath}
0\to A\to Q_0\to Q_1\to\dots\to Q_k\to 0, 
\end{displaymath}
such that each $Q_i\in\mathrm{add}(T)$ and $\mathrm{proj.dim}(Q_i)\leq i$, 
for all $i=0,1,\dots,k$.

Since the characteristic tilting module is a
(generalized) tilting module, Auslander-Ringel-Gorenstein property agrees with
$T$-regularity in the terminology of Subsection~\ref{s1.5}.

\section{Regularity phenomena in $\mathcal{S}$-subcategories in
$\mathcal{O}$}\label{s6}

\subsection{$\mathcal{S}$-subcategories in $\mathcal{O}$}\label{s6.1}

We again fix a parabolic subalgebra $\mathfrak{p}$ of $\mathfrak{g}$ containing
$\mathfrak{h}\oplus\mathfrak{n}_+$ and restrict our attention to the integral
part $\mathcal{O}_{\mathrm{int}}$ of $\mathcal{O}$. 

Let $\mathcal{X}$ denote the Serre subcategory of $\mathcal{O}_{\mathrm{int}}$
generated by all simple highest weight modules whose highest weights $\lambda$ are not
anti-dominant with respect to $W_{\mathfrak{p}}$, that is, $w\cdot\lambda< \lambda$ for some $w\in W_{\mathfrak p}$. Denote by
$\mathcal{S}^{\mathfrak{p}}$ the Serre quotient category
$\mathcal{O}_{\mathrm{int}}/\mathcal{X}$, see \cite{FKM,MS}.
From \cite{FKM}, we know that blocks of $\mathcal{S}^{\mathfrak{p}}$
correspond to 
strongly standardly stratified algebras.

Let $\mathcal{S}^{\mathfrak{p}}_0$ be the principal block of $\mathcal{S}^{\mathfrak{p}}$.

\subsection{$\mathcal{S}^{\mathfrak{p}}_0$ is Auslander-Ringel-Gorenstein}\label{s6.2}

\begin{theorem}\label{thm11}
The category  $\mathcal{S}^{\mathfrak{p}}_0$ is Auslander-Ringel-Gorenstein.
\end{theorem}

\begin{proof}
By \cite{FKM}, the indecomposable projectives in  $\mathcal{S}^{\mathfrak{p}}_0$
are exactly the images of $P_w$, where $w$ belongs to  
$\mathrm{long}({}_{W_{\mathfrak{p}}}\hspace{-2mm}\setminus\hspace{-1mm} W)$ (the set of longest coset representatives in ${}_{W_{\mathfrak{p}}}\hspace{-2mm}\setminus\hspace{-1mm} W$).
Furthermore, the indecomposable tilting objects in $\mathcal{S}^{\mathfrak{p}}_0$
are exactly the images of $T_w$, where 
$w\in \mathrm{short}({}_{W_{\mathfrak{p}}}\hspace{-2mm}\setminus\hspace{-1mm} W)$.

Note that the above objects in $\mathcal{O}$ are exactly those indecomposable
projective (resp. tilting) objects which are admissible in the sense of
\cite[Lemma~14]{PW}. From \cite[Lemma~14 and Theorem~15]{PW} it follows that 
the minimal projective resolution (in $\mathcal{O}$) of any $T_w$ as above
contains only $P_x$ as above. The Ringel dual of this property is that a minimal
tilting coresolution (in $\mathcal{O}$) of any $P_x$ as above
contains only $T_w$ as above. Since the projection functor
$\mathcal{O}_0\tto \mathcal{S}^{\mathfrak{p}}_0$ is exact
and preserves the projective dimension for the involved projective and tilting
modules, see \cite[Theorem~15]{PW}, 
the claim of our theorem follows from Theorem~\ref{thm1}.
\end{proof}

\subsection{$\mathcal{S}^{\mathfrak{p}}_0$ is Auslander-Gorenstein}\label{s6.3}

\begin{theorem}\label{thm12}
The category  $\mathcal{S}^{\mathfrak{p}}_0$ is Auslander-Gorenstein.
\end{theorem}

\begin{proof}
The indecomposable injectives in $\mathcal S^{\mathfrak p}_0$ are exactly the images of $I_w$ for $w\in \mathrm{long}({}_{W_{\mathfrak{p}}}\hspace{-2mm}\setminus\hspace{-1mm} W)$ and these $I_w\in \cO$ are admissible in the sense of \cite{PW}. 
Thus, the claim follows from Theorem~\ref{thm3} similarly to the proof of Theorem \ref{thm11}.
\end{proof}

\subsection{Singular blocks}\label{s6.4}

\begin{theorem}\label{thm14}
All blocks of $\mathcal{S}^{\mathfrak{p}}$ are both 
Auslander-Ringel-Gorenstein and Auslander-Gorenstein.
\end{theorem}

\begin{proof}
Mutatis mutandis the proof of Corollary~\ref{cor6}
\end{proof}

\subsection{$\mathfrak{sl}_3$-example}\label{s6.5}

For the Lie algebra $\mathfrak{sl}_3$, we have
$W=\{e,s,t,st,ts,w_0=sts=tst\}$. Assume that $W_\mathfrak{p}=\{e,s\}$,
then  $\mathrm{long}({}_{W_{\mathfrak{p}}}\hspace{-2mm}\setminus
\hspace{-1mm} W)=\{s,st,w_0\}$.
The projective dimensions
of the indecomposable tilting and injective modules in 
$\mathcal{S}_0^{\mathfrak{p}}$ 
are given by:
\begin{displaymath}
\begin{array}{c||c|c|c}
w:&e&t&ts\\ \hline
\mathrm{proj.dim}(T_w):&0&1&1
\end{array}\qquad
\begin{array}{c||c|c|c}
w:&s&st&w_0\\ \hline
\mathrm{proj.dim}(I_w):&2&2&0
\end{array}
\end{displaymath}
The minimal (ungraded) tilting coresolutions of the
{\color{blue}indecomposable projectives} in $\mathcal{S}_0^{\mathfrak{p}}$ are:
\begin{displaymath}
0\to {\color{blue}P_{s} }\to 
T_e\to T_t\to 0,
\end{displaymath}
\begin{displaymath}
0\to {\color{blue}P_{st}}\to T_{e}\to T_{ts}\to 0,
\end{displaymath}
\begin{displaymath}
0\to {\color{blue}P_{w_0}}\to T_{e}\to 0,
\end{displaymath}
The minimal (ungraded) injective coresolutions of the
{\color{blue}indecomposable projectives} in $\mathcal{S}_0^{\mathfrak{p}}$ are:
\begin{displaymath}
0\to {\color{blue}P_{s}}\to I_{w_0}\to I_{w_0}\to  I_s \to 0,
\end{displaymath}
\begin{displaymath}
0\to {\color{blue}P_{st}}\to I_{w_0}\to I_{w_0}\to  I_{st} \to 0,
\end{displaymath}
\begin{displaymath}
0\to {\color{blue}P_{w_0}}\to I_{w_0} \to 0,
\end{displaymath}

\section{Applications to the cohomology of twisting and Serre functors}\label{s7}

\subsection{Twisting and Serre functors on $\mathcal{O}$}\label{s7.1}

For a simple reflection $s$, we denote by $\top_s$ the corresponding twisting 
functor on $\mathcal{O}$, see \cite{AS}. For $w\in W$, with a fixed reduced
expression $w=s_1s_2\cdots s_k$, we denote by $\top_w$ the composition
$\top_{s_1}\top_{s_2}\cdots\top_{s_k}$ and note that it does
not depend on the choice of a reduced expression by \cite{KM}.

All functors $\top_w$ are right exact, functorially commute with 
projective functors, acyclic on Verma modules
and the corresponding derived functors are self-equivalences of the
derived category of $\mathcal{O}$. Furthermore, we have
$\top_{w_0}P_x\cong T_{w_0x}$ and $\top_{w_0}T_x\cong I_{w_0x}$,
for all $x\in W$. We refer to \cite{AS,KM} for all details.

The functor $(\mathcal{L}\top_{w_0})^2$ is a Serre functor on 
$\mathcal{D}^b(\mathcal{O}_0)$, see \cite{MS2}.

\subsection{Auslander regularity via Serre functors}\label{s7.2}

Let $A$ be a finite dimensional associative algebra of finite global dimension
over an algebraically closed field $\Bbbk$. Then the left derived
$\mathcal{L}\mathbf{N}$ of 
the Nakayama functor $\mathbf{N}=A^*\otimes_A{}_-$ for $A$ is a Serre functor on 
$\mathcal{D}^b(A)$.

Recall that $L_i$, where $i=1,2,\dots,k$, is a complete and irredundant
list of simple $A$-modules, $P_i$ denotes the indecomposable projective cover of $L_i$
and $I_i$ denotes the indecomposable injective envelope of $L_i$.
Let $P$ be a basic projective generator of $A$-mod and 
$I$ a basic injective cogenerator of $A$-mod.

\begin{lemma}\label{lem21}
For $M\in A$-mod, $j\in\{1,2,\dots,k\}$ and $i\in\mathbb{Z}_{\geq 0}$, we have 
\begin{displaymath}
\dim \mathrm{Ext}^i_A(M,P_j)= (\mathcal{L}_i \mathbf{N}(M):L_j). 
\end{displaymath}
\end{lemma}

\begin{proof}
Being a Serre functor, $\mathcal{L}\mathbf{N}$ is a self-equivalence of 
$\mathcal{D}^b(A)$. Therefore, we have
\begin{displaymath}
\begin{array}{rcl}
\mathrm{Ext}^i_A(M,P_j)&=&\mathrm{Hom}_{\mathcal{D}^b(A)}(M,P_j[i])\\
&=&\mathrm{Hom}_{\mathcal{D}^b(A)}(\mathcal{L} \mathbf{N}(M),\mathcal{L} \mathbf{N}(P_j[i]))\\
&=&\mathrm{Hom}_{\mathcal{D}^b(A)}(\mathcal{L} \mathbf{N}(M),I_j[i]).
\end{array}
\end{displaymath}
The claim of the lemma follows.
\end{proof}

The above observation has the following consequence:

\begin{proposition}\label{prop22}
The algebra $A$ is Auslander regular if and only if, 
for any simple $A$-module $L_j$,
we have $\mathcal{L}_i\mathbf{N}(L_j)=0$, 
for all $i<\mathrm{proj.dim}(I_j)$.
\end{proposition}

\begin{proof}
By definition, $A$ is Auslander regular if and only if, 
for any simple $A$-module $L_j$, we have 
$\mathrm{Ext}^i(L_j,A)=0$ unless $i\geq \mathrm{proj.dim}(I_j)$.
Now the necessary claim follows from Lemma~\ref{lem21}.
\end{proof}

\subsection{Cohomology of twisting and Serre 
functors for category $\mathcal{O}_0$}\label{s7.3}

\begin{corollary}\label{cor23}
For $w\in W$, we have $(\mathcal{L}_i\top_{w_0})^2L_w=0$, for all $0\leq i<2\mathbf{a}(w_0w)$.
\end{corollary}

\begin{proof}
By Theorem~\ref{thm3}, $\mathcal{O}_0$ is Auslander regular. By
the main results of \cite{Ma3,Ma4}, the projective dimension of
$I_w$ equals $2\mathbf{a}(w_0w)$. Therefore the claim 
follows  from Proposition~\ref{prop22}.
\end{proof}

Corollary~\ref{cor23} admits the following refinement.

\begin{proposition}\label{prop24}
For $w\in W$, we have  $\mathcal{L}_i\top_{w_0}L_w=0$, for all $0\leq i<\mathbf{a}(w_0w)$.
\end{proposition}

\begin{proof}
The injective resolution $\mathcal{I}_\bullet(L_{w_0})$ of $T_{w_0}=L_{w_0}$ 
is linear and is a dominant standard object 
in the category of linear complexes of injective modules in $\mathcal{O}_0$, by 
the Koszul self-duality of $\mathcal{O}_0$, see \cite{So}. Therefore, for $x\in W$,
the module $I_x$ can only appear as a summand of $\mathcal{I}_i(L_{w_0})$, for
$\mathbf{a}(w_0x\inv)=\mathbf{a}(w_0x)\leq i$. This means that
\begin{displaymath}
\mathrm{Ext}^i_{\mathcal{O}}(L_w,T_{w_0})=0,\text{ for all } i< \mathbf{a}(w_0w).
\end{displaymath}

Note that, for any projective functor $\theta$, all simple subquotients $L_x$
of the module $\theta L_w$ satisfy $\mathbf{a}(w_0x)\geq \mathbf{a}(w_0w)$. Therefore,
for the adjoint $\theta'$ of $\theta$, the previous
paragraph implies that
\begin{displaymath}
\mathrm{Ext}^i_{\mathcal{O}}(L_w,\theta T_{w_0})=
\mathrm{Ext}^i_{\mathcal{O}}(\theta' L_w, T_{w_0})=0,
\text{ for all } i< \mathbf{a}(w_0w).
\end{displaymath}

To sum up, for any tilting module $T$, we have
\begin{displaymath}
\mathrm{Ext}^i_{\mathcal{O}}(L_w,T)=0,
\text{ for all } i< \mathbf{a}(w_0w).
\end{displaymath}
Applying the equivalence $\mathcal{L}\top_{w_0}$ and noting that it sends
tilting modules to injective, we obtain the claim of the proposition.
\end{proof}

Now we prove a result ``in the opposite direction''.
Let $I$ be an injective cogenerator of $\mathcal{O}_0$.

\begin{proposition}\label{prop27}
For $w\in W$, we have  $\mathcal{L}_i\top_{w_0}L_w=0$, for all $i>\ell(w_0w)$.
\end{proposition}

\begin{proof}
We want to prove that
$\mathrm{Hom}_{\mathcal{D}^b(\mathcal{O})}(\mathcal{L}\top_{w_0}L_w,I[i])=0$, for all $i>\ell(w_0w)$.
Applying the adjoint of the equivalence $\mathcal{L}\top_{w_0}$, we get an equivalent 
statement that $\mathrm{Hom}_{\mathcal{D}^b(\mathcal{O})}(L_w,T[i])=0$, for all $i>\ell(w_0w)$,
where $T$ is the characteristic tilting module in $\cO_0$.

Consider the linear complex $\mathcal{T}_\bullet(L_w)$ of tilting modules which represents
$L_w$. By \cite{Ma}, it is a tilting object in the category of linear complexes of tilting modules.
Combining the Ringel and Koszul self-dualities of $\mathcal{O}_0$, we obtain that the absolute
value of the minimal non-zero component of $\mathcal{T}_\bullet(L_w)$ equals the maximal
degree of a non-zero component of $T_{w_0w\inv w_0}$. The latter is equal to $\ell(w_0w)$. Now the
necessary claim follows from \cite[Chapter III(2), Lemma 2.1]{Ha}.
\end{proof}

\begin{proposition}\label{prop25}
For $w\in W$, we have  $[\mathcal{L}_i\top_{w_0} I:L_w]\neq 0$ 
only if $i\leq \mathbf{a}(w_0w)$.
\end{proposition}

\begin{proof}
Applying projective functors, the statement reduces to the special case
when $I$ is substituted by $I_e=\nabla_e$.  Note that
$[\mathcal{L}_i\top_{w_0} \nabla_e:L_w]$ equals the dimension
of $\mathrm{Hom}_{\mathcal{D}^b(\mathcal{O})}(\mathcal{L}\top_{w_0} \nabla_e,I_w[i])$.

Now, we write $\nabla_e= \mathcal{L}\top_{w_0} T_{w_0}$. Moving 
$(\mathcal{L}\top_{w_0})^2$ from the first argument to the second using
adjunction, we arrive to the space
$\mathrm{Hom}_{\mathcal{D}^b(\mathcal{O})}(T_{w_0},P_w[i])$.
Now the necessary claim follows from the observation that
$\mathbf{r}(P_w)=\mathbf{a}(w_0w)$, which is the Ringel dual of the
main results of \cite{Ma3,Ma4}.
\end{proof}

\subsection{$\mathfrak{sl}_3$-example}\label{s7.4}

In the case of $\mathfrak{sl}_3$, we have $W=\{e,s,t,st,ts,w_0\}$.
In Figure~\ref{fig1}, we give an explicit $\mathbb{Z}$-graded description
of composition factors of the tilting resolution
\begin{displaymath}
T_{w_0}\hookrightarrow T_{st}\oplus T_{ts}\to T_s\oplus T_t\to T_e 
\end{displaymath}
of $\nabla_e$ and its image after applying $\mathcal{L}\top_{w_0}$.
The original resolution is in {\color{magenta}magenta} and black
with $\nabla_e$ being the {\color{magenta}magenta} part.
The simple subquotients added during the application of 
$\mathcal{L}\top_{w_0}$ are {\color{blue}blue}.
The resulting cohomology in negative positions is \fbox{boxed}.
The module $L_w$ is denoted by $w$. The values of the $\mathbf{a}$-function
are as follows: $\mathbf{a}(e)=0$,
$\mathbf{a}(s)=\mathbf{a}(t)=\mathbf{a}(st)=\mathbf{a}(ts)=1$,
$\mathbf{a}(w_0)=3$.

\begin{figure}\label{fig1}
\resizebox{\textwidth}{!}{$$
\xymatrix@C=0.1em@R=0.1em{
&&&&&&&&&&&&&&&&&&&&&&&&&&&&&&{\color{magenta}w_0}&&\\
&&&&&&&&&&&&&&&&&&&&&&&&&&&&&{\color{magenta}st}&&{\color{magenta}ts}&\\
&&&&&&&&&&&&&&&&&&w_0&&&&&&w_0&&&&w_0&{\color{magenta}s}&&{\color{magenta}t}&w_0\\
&&&&&&&&&&&&&&&&&st&&ts&&&&st&&ts&&&st&ts&{\color{magenta}e}&st&ts\\
&&&\hookrightarrow&&w_0&&&&\oplus&&w_0&&&&\to&
\fbox{${\color{blue}t}$}&w_0&s&w_0&&\oplus&
\fbox{${\color{blue}s}$}&w_0&t&w_0&&\to&w_0&s&&t&w_0\\
&&&&&st&&{\color{blue}ts}&&&&ts&&{\color{blue}st}
&&&\fbox{${\color{blue}st}$}&ts&\fbox{${\color{blue}e}$}&st&{\color{blue}ts}
&&\fbox{${\color{blue}ts}$}&ts&\fbox{${\color{blue}e}$}&st&{\color{blue}st}&&&st&&ts&\\
&w_0&&&&w_0&{\color{blue}s}&&{\color{blue}t}&&&
w_0&{\color{blue}t}&&{\color{blue}s}&&&
{\color{blue}s}&w_0&{\color{blue}t}&&&&
{\color{blue}t}&w_0&{\color{blue}s}&&&&&w_0&&\\
{\color{blue}st}&&{\color{blue}ts}&&
{\color{blue}st}&&{\color{blue}ts}&\fbox{${\color{blue}e}$}&&&
{\color{blue}st}&&{\color{blue}ts}&\fbox{${\color{blue}e}$}&&&&&
{\color{blue}st}&&&&&&{\color{blue}ts}&&&&&&&&\\
{\color{blue}s}&&{\color{blue}t}&&&{\color{blue}s}&&&&&&{\color{blue}t}&&&&&&&&&&&&&&&&&&&&&\\
&\fbox{${\color{blue}e}$}&&&&&&&&&&&&&&&&&&&&&&&&&&&&&&&\\
}
$$}
\caption{$\mathcal{L}\top_{w_0}\nabla_e$ and its cohomology for $\mathfrak{sl}_3$}
\end{figure}
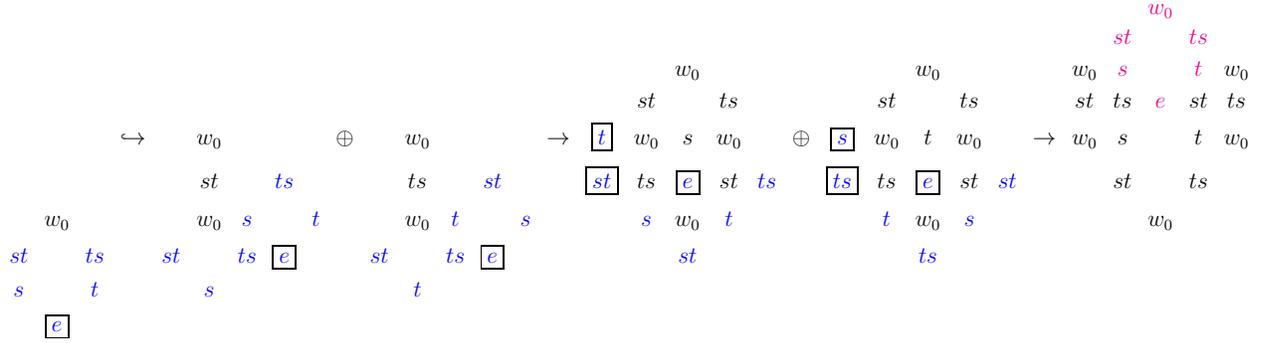

\section{Regularity phenomena with respect to twisted projective 
and tilting modules}\label{s8}

\subsection{Twisted projective modules}\label{s8.1}

Let $P$ be a projective generator of $\mathcal{O}_0$.
For $w\in W$, the module $\top_w P$ is a (generalized) tilting
module in $\mathcal{O}_0$ because $\top_w$ is a derived 
self-equivalence which is acyclic on modules with Verma flag.
A question is, for which
$w$ is the category $\mathcal{O}_0$ $\top_w P$-regular.
Below we show that the answer is non-trivial.

\subsection{Regularity with respect to twisted projectives}\label{s8.2}

\begin{theorem}\label{thm8.n1}
If $w=w_0^\mathfrak{p}$, for some parabolic subalgebra $\mathfrak{p}$
in $\mathfrak{g}$,
then $\mathcal{O}_0$ is $\top_w P$-regular.
\end{theorem}

\begin{proof}
Let $w=\wi$ as above.
Since twisting functors functorially commute with projective functors,
we only need to show that $\Delta_e$ has a coresolution by modules
in $\mathrm{add}(\top_w P)$ satisfying the regularity condition.

By construction, twisting functors commute with parabolic induction.
For the category $\mathcal{O}$ associated to the Levi subalgebra 
$\mathfrak{l}$ of 
$\mathfrak{p}$, the claim of our Theorem coincides with the claim of 
Theorem~\ref{thm1}. The parabolic induction from $\mathfrak l$
to $\mathfrak{g}$ is exact and sends the indecomposable projective $P_x^{\mathfrak l}$ (for $x\in W_{\mathfrak p}$)
to the indecomposable projective $P_x$. 
It also sends (indecomposable) tiltings
to our twisted projective modules. To see this, 
write the indecomposable tilting module for $\mathfrak l$   corresponding to $x\in W_{\mathfrak p}$ as  $T^{\mathfrak l}_x \cong \top_wP^{\mathfrak l}_{wx}$ and use that the parabolic induction commutes with $\top_w$ to conclude that $T^{\mathfrak l}_x$ is sent to $\top_w P_{wx}$.
Therefore a tilting coresolution of
the dominant projective for $\mathfrak{l}$  is sent to a coresolution
of the dominant projective for $\mathfrak{g}$ by our twisted projective modules. The claim follows.
\end{proof}

\begin{corollary}\label{cor8.n2}
If $w=w_0^\mathfrak{p}$, for some parabolic subalgebra $\mathfrak{p}$
in $\mathfrak{g}$,
then all blocks of $\mathcal{O}$ are $\top_w P$-regular.
\end{corollary}

\begin{proof}
Since twisting functors functorially commute with projective functors,
 we can use translations to walls to extend Theorem~\ref{thm8.n1} to singular blocks.
\end{proof}

\begin{remark}\label{ex8.n3}
{\em \em 
The module $\Delta_w$ admits a (linear) coresolution by tilting modules, 
which starts with $T_w$. Applying the inverse of $\mathcal{L}\top_w$ 
to this coresolution, we obtain a coresolution of $\Delta_e$ by
modules in $\mathrm{add}(\top_{w^{-1}w_0}P)$. We note that, 
by \cite{AS}, the inverse of $\mathcal{L}\top_w$ is
$\mathcal{R}(\star\circ\top_{w^{-1}}\circ\star)$, where $\star$ is the simple preserving duality, and the claim in the
previous sentence follows by using the acyclicity results in \cite{AS}.
Hence, a necessary condition for
$\mathcal{O}_0$ to be $\top_{xw_0}P$-regular is that the module
$(\mathcal{L}\top_w)^{-1}T_w$, which starts this coresolution,
is projective. In case the multiplicity
of $\Delta_{w_0}$ in a standard filtration of $T_w$ is greater than $1$,
the module $(\mathcal{L}\top_w)^{-1}T_w$ will have $\Delta_e$ appearing
with multiplicity $1$ (as $\Delta_w$ appears in $T_w$ with multiplicity $1$)
while some standard module will have higher multiplicity, by assumption.
Therefore, in this case, $(\mathcal{L}\top_w)^{-1}T_w$ is not a projective module.
This shows that the condition $[T_w:\Delta_{w_0}]=1$ is necessary
for $\mathcal{O}_0$ to be  $\top_{xw_0}P$-regular. This implies that
examples of $w\in W$ such that  $\mathcal{O}_0$ is not  $\top_{w}P$-regular
exist already in type $A_3$. We will see in Subsection \ref{s8.5} below that $\top_{w}P$-regularity can fail already in $\mathcal{O}_0$ of type  $A_2$.
}
\end{remark}

\subsection{Twisted tilting modules}\label{s8.3}

Let $T$ be a characteristic tilting module for $\mathcal{O}_0$.
For $w\in W$, the module $\top_w T$ is a (generalized) tilting
module in $\mathcal{O}_0$ because $\top_w$ is a derived 
self-equivalence which is acyclic on modules with Verma flag.

This raises an interesting problem, namely, to determine for which
$w$ the category $\mathcal{O}_0$ is $\top_w T$-regular.
We show below that the answer is non-trivial.

\subsection{Regularity with respect to twisted tiltings}\label{s8.4}

\begin{theorem}\label{thm8.n21}
If $w=w_0^\mathfrak{p}$, for some parabolic subalgebra $\mathfrak{p}$
in $\mathfrak{g}$,
then $\mathcal{O}_0$ is $\top_w T$-regular.
\end{theorem}

\begin{proof}

As usual, we use that the projective functors commutes with twisting functors to reduce the claim to finding a desired coresolution for $P_e=\Delta_e$.

Let $\mathfrak l$ be the Levi subalgebra of $\mathfrak p$ and take a coresolution of $\Delta_e^{\mathfrak l}=P^{\mathfrak l}_e$ by injectives for $\mathfrak l$ with the regularity property, guaranteed by Theorem \ref{thm3}.
Then just like in the proof of Theorem~\ref{thm8.n1}, the parabolic induction produces a coresolution of $\Delta_e$ in $\operatorname{add}(\top_w T)$ with the regularity condition. 
In fact, the $w_0^\mathfrak{p}$-twists of tiltings are obtained by the parabolic induction from injective modules over 
$\mathfrak{l}$, which are the $\wi$-twists of tiltings over $\mathfrak l$. 
The proof is complete.
\end{proof}

\begin{corollary}\label{cor8.n22}
If $w=w_0^\mathfrak{p}$, for some parabolic subalgebra $\mathfrak{p}$
in $\mathfrak{g}$,
then all blocks of $\mathcal{O}$ are $\top_w T$-regular.
\end{corollary}

\begin{proof}
Since twisting functors functorially commute with projective functors,
we can use translations to walls to extend the statement in 
Theorem \ref{thm8.n21} to singular blocks.
\end{proof}

We will see in Subsection \ref{s8.5} below that $\top_{w}T$-regularity can fail already in $\mathcal{O}_0$ of type $A_2$.

\subsection{$\mathfrak{sl}_3$-example}\label{s8.5}

For the Lie algebra $\mathfrak{sl}_3$, we have
$W=\{e,s,t,st,ts,w_0=sts=tst\}$. 

The left of the two tables below describes the projective dimensions 
of the twisted projective modules $\top_x P_y$. The right table
below describes the projective dimensions 
of the twisted tilting modules $\top_x T_y$. 
\begin{displaymath}
\begin{array}{c||c|c|c|c|c|c}
x\backslash y&e&s&t&st&ts&w_0\\ \hline\hline
e&0&0&0&0&0&0\\\hline
s&1&0&1&0&1&0\\\hline
t&1&1&0&1&0&0\\\hline
st&2&1&1&1&1&0\\\hline
ts&2&1&1&1&1&0\\\hline
w_0&3&1&1&1&1&0
\end{array}\qquad\qquad
\begin{array}{c||c|c|c|c|c|c}
x\backslash y&e&s&t&st&ts&w_0\\ \hline\hline
e&0&1&1&1&1&3\\\hline
s&0&2&1&2&1&4\\\hline
t&0&1&2&1&2&4\\\hline
st&0&2&2&2&2&5\\\hline
ts&0&2&2&2&2&5\\\hline
w_0&0&2&2&2&2&6
\end{array} 
\end{displaymath}

Here are the graded characters of the modules $\top_s P_x$
(with the characters of the tilting cores displayed in {\color{magenta}magenta}):

\resizebox{\textwidth}{!}{
$$
\xymatrix@C=0.1em@R=0.1em{
\mathrm{deg}\backslash {\color{blue}x}&&&
{\color{blue}e}&&&&{\color{blue}s}&&&&{\color{blue}t}&&&&
{\color{blue}st}&&&&&{\color{blue}ts}&&&&&&{\color{blue}w_0}&&&\\
-1&|&&&&|&&s&&|&&&|&&&st&&&|&&&&&|&&&{\color{magenta}w_0}&&&\\
0&|&&s&&|&st&e&ts&|&st&&|&&s&{\color{magenta}w_0}&t&&|&&{\color{magenta}w_0}&s&&|&&{\color{magenta}st}&&{\color{magenta}ts}&&\\
1&|&st&&ts&|&s&{\color{magenta}w_0}&t&|&s&{\color{magenta}w_0}&|&st&{\color{magenta}ts}&e&{\color{magenta}st}&ts&|&{\color{magenta}st}&e&{\color{magenta}ts}&st&|&{\color{magenta}w_0}&{\color{magenta}s}&&{\color{magenta}t}&{\color{magenta}w_0}&\\
2&|&&{\color{magenta}w_0}&&|&&{\color{magenta}st}&ts&|&st&{\color{magenta}ts}&|&&{\color{magenta}w_0}&{\color{magenta}s}&{\color{magenta}w_0}&t&|&{\color{magenta}w_0}&{\color{magenta}t}&{\color{magenta}w_0}&s&|&{\color{magenta}st}&{\color{magenta}ts}&{\color{magenta}e}&{\color{magenta}st}&{\color{magenta}ts}&\\
3&|&&&&|&&{\color{magenta}w_0}&&|&&{\color{magenta}w_0}&|&&{\color{magenta}ts}&&{\color{magenta}st}&&|&{\color{magenta}st}&&{\color{magenta}ts}&&|&{\color{magenta}w_0}&{\color{magenta}s}&&{\color{magenta}t}&{\color{magenta}w_0}&\\
4&|&&&&|&&&&|&&&|&&&{\color{magenta}w_0}&&&|&&{\color{magenta}w_0}&&&|&&{\color{magenta}st}&&{\color{magenta}ts}&&\\
5&|&&&&|&&&&|&&&|&&&&&&|&&&&&|&&&{\color{magenta}w_0}&&&
}
$$
}

Here are the graded characters of the modules $\top_{ts} P_x$
(with the characters of the tilting cores displayed in {\color{magenta}magenta}):

\resizebox{\textwidth}{!}{
$$
\xymatrix@C=0.1em@R=0.1em{
\mathrm{deg}\backslash {\color{blue}x}&&
{\color{blue}e}&&&{\color{blue}s}&&{\color{blue}t}&&&
{\color{blue}st}&&&&&{\color{blue}ts}&&&&&{\color{blue}w_0}&&\\
-2&|&&|&&&|&&|&&&&&|&&&&|&&&{\color{magenta}w_0}&&&\\
-1&|&&|&ts&&|&&|&&{\color{magenta}w_0}&t&&|&&{\color{magenta}w_0}&&|&&{\color{magenta}st}&&{\color{magenta}ts}&&\\
0&|&ts&|&t&{\color{magenta}w_0}&|&{\color{magenta}w_0}&|&{\color{magenta}ts}&e&{\color{magenta}st}&ts&|&{\color{magenta}st}&&{\color{magenta}ts}&|&{\color{magenta}w_0}&{\color{magenta}s}&&{\color{magenta}t}&{\color{magenta}w_0}&\\
1&|&{\color{magenta}w_0}&|&ts&{\color{magenta}st}&|&{\color{magenta}ts}&|&{\color{magenta}w_0}&{\color{magenta}s}&{\color{magenta}w_0}&t&|&{\color{magenta}w_0}&{\color{magenta}t}&{\color{magenta}w_0}&|&{\color{magenta}st}&{\color{magenta}ts}&{\color{magenta}e}&{\color{magenta}st}&{\color{magenta}ts}&\\
2&|&&|&&{\color{magenta}w_0}&|&{\color{magenta}w_0}&|&{\color{magenta}st}&&{\color{magenta}ts}&&|&{\color{magenta}st}&&{\color{magenta}ts}&|&{\color{magenta}w_0}&{\color{magenta}s}&&{\color{magenta}t}&{\color{magenta}w_0}&\\
3&|&&|&&&|&&|&&{\color{magenta}w_0}&&&|&&{\color{magenta}w_0}&&|&&{\color{magenta}st}&&{\color{magenta}ts}&&\\
4&|&&|&&&|&&|&&&&&|&&&&|&&&{\color{magenta}w_0}&&&
}
$$
}

Here are the graded characters of the modules $\top_s T_x$
(with the characters of the tilting cores displayed in {\color{magenta}magenta}):

\resizebox{\textwidth}{!}{
$$
\xymatrix@C=0.1em@R=0.1em{
\mathrm{deg}\backslash {\color{blue}x}&&
{\color{blue}w_0}&&{\color{blue}st}&&&{\color{blue}ts}&&&
{\color{blue}s}&&&&&{\color{blue}t}&&&&&{\color{blue}e}&&&\\
-4&|&&|&&&|&&|&&&&&|&&&&|&&&{\color{magenta}w_0}&&&\\
-3&|&&|&&&|&&|&&{\color{magenta}w_0}&&&|&&{\color{magenta}w_0}&&|&&{\color{magenta}st}&&{\color{magenta}ts}&&\\
-2&|&&|&{\color{magenta}w_0}&&|&{\color{magenta}w_0}&|&{\color{magenta}st}&&{\color{magenta}ts}&&|&{\color{magenta}st}&&{\color{magenta}ts}&|&{\color{magenta}w_0}&{\color{magenta}s}&&{\color{magenta}t}&{\color{magenta}w_0}&\\
-1&|&{\color{magenta}w_0}&|&{\color{magenta}st}&ts&|&{\color{magenta}ts}&|&{\color{magenta}w_0}&{\color{magenta}s}&{\color{magenta}w_0}&t&|&{\color{magenta}w_0}&{\color{magenta}t}&{\color{magenta}w_0}&|&{\color{magenta}st}&{\color{magenta}ts}&{\color{magenta}e}&{\color{magenta}st}&{\color{magenta}ts}&\\
0&|&ts&|&{\color{magenta}w_0}&t&|&{\color{magenta}w_0}&|&{\color{magenta}st}&ts&{\color{magenta}ts}&e&|&{\color{magenta}st}&&{\color{magenta}ts}&|&{\color{magenta}w_0}&{\color{magenta}s}&&{\color{magenta}t}&{\color{magenta}w_0}&\\
1&|&&|&&ts&|&&|&&{\color{magenta}w_0}&t&&|&&{\color{magenta}w_0}&&|&&{\color{magenta}st}&&{\color{magenta}ts}&&\\
2&|&&|&&&|&&|&&&&&|&&&&|&&&{\color{magenta}w_0}&&&
}
$$
}

Here are the graded characters of the modules  $\top_{ts} T_x$
(with the characters of the tilting cores displayed in {\color{magenta}magenta}):

\resizebox{\textwidth}{!}{
$$
\xymatrix@C=0.1em@R=0.1em{
\mathrm{deg}\backslash {\color{blue}x}&&&
{\color{blue}w_0}&&&&{\color{blue}st}&&&&{\color{blue}ts}&&&
{\color{blue}s}&&&&&&{\color{blue}t}&&&&&&{\color{blue}e}&&&\\
-5&|&&&&|&&&&|&&&|&&&&&&|&&&&&|&&&{\color{magenta}w_0}&&&\\
-4&|&&&&|&&&&|&&&|&&{\color{magenta}w_0}&&&&|&&{\color{magenta}w_0}&&&|&&{\color{magenta}st}&&{\color{magenta}ts}&&\\
-3&|&&&&|&{\color{magenta}w_0}&&&|&&{\color{magenta}w_0}&|&{\color{magenta}st}&&{\color{magenta}ts}&&&|&{\color{magenta}st}&&{\color{magenta}st}&&|&{\color{magenta}w_0}&{\color{magenta}s}&&{\color{magenta}t}&{\color{magenta}w_0}&\\
-2&|&&{\color{magenta}w_0}&&|&{\color{magenta}st}&ts&&|&st&{\color{magenta}ts}&|&{\color{magenta}w_0}&{\color{magenta}s}&{\color{magenta}w_0}&t&&|&{\color{magenta}w_0}&{\color{magenta}t}&{\color{magenta}w_0}&s&|&{\color{magenta}st}&{\color{magenta}ts}&{\color{magenta}e}&{\color{magenta}st}&{\color{magenta}ts}&\\
-1&|&st&&ts&|&{\color{magenta}w_0}&t&s&|&s&{\color{magenta}w_0}&|&{\color{magenta}st}&ts&{\color{magenta}ts}&e&st&|&{\color{magenta}st}&e&{\color{magenta}ts}&st&|&{\color{magenta}w_0}&{\color{magenta}s}&&{\color{magenta}t}&{\color{magenta}w_0}&\\
0&|&&s&&|&ts&st&e&|&st&&|&&{\color{magenta}w_0}&t&s&&|&&{\color{magenta}w_0}&s&&|&&{\color{magenta}st}&&{\color{magenta}ts}&&\\
1&|&&&&|&&s&&|&&&|&&&st&&&|&&&&&|&&&{\color{magenta}w_0}&&&
}
$$
}

The cases $x=e$ and $w_0$ are already discussed in the previous sections.
To prove regularity, we only need to consider the coresolution of 
{\color{blue}$P_e$}.
Up to the symmetry of the Dynkin diagram, it is enough to consider the
four cases $\top_sP$, $\top_{ts}P$, $\top_sT$ and $\top_{ts}T$.
The first two are given as follows:
\begin{displaymath}
0\to{\color{blue}P_e}\to \top_s P_s\to \top_s P_e\to 0
\end{displaymath}
\begin{displaymath}
0\to{\color{blue}P_e}\to \top_{ts} P_{st}\to \top_{ts} P_s\oplus
\top_{ts}P_t\to\top_{ts}P_e\to 0
\end{displaymath}
Here we see that the first coresolution is regular,
while in the second one, $\top_{ts} P_{st}$ is not projective and
hence we do not have regularity with respect to $\top_{ts} P$.
The case of $\top_sT$ is regular and given as follows:
\begin{displaymath}
0\to{\color{blue}P_e}\to 
\top_s T_e\to 
\top_s T_t\oplus \top_s T_{ts}\oplus \top_s T_e\to
\top_s T_{ts}\oplus  \top_s T_t\oplus \top_s T_s\to
\top_s T_{ts}\oplus \top_s T_{st}\to \top_s T_{w_0}\to 0
\end{displaymath}
Finally, we claim that we do not have the regularity in
the case of $\top_{ts}T$. Indeed, in order not to fail already
in position zero, we must start with 
$0\to{\color{blue}P_e}\to \top_{ts}T_e\to\mathrm{Coker}$.
Further, in order not to fail on the next step, we again must
embed $\mathrm{Coker}$ into $\top_{ts}T_e\oplus \top_{ts}T_e$.
The new cokernel will necessarily have both $L_s$ and $L_t$ in the
socle. However, $L_t$ does not appear in the socle of $\top_{ts}T$
and hence the coresolution cannot continue. This implies that one
of the first two steps requires correction by adding non-projective
summands of $\top_{ts}T$, which implies the failure of the regularity.

\section{Regularity phenomena with respect to shuffled projective 
and tilting modules}\label{s9}

\subsection{Shuffled projective modules}\label{s9.1}

For $w\in W$, we denote by $\mathrm{C}_w$ the corresponding shuffling
functor on $\mathcal{O}_0$, see \cite[Section~5]{MS}.
Let $P$ be a projective generator of $\mathcal{O}_0$.
For $w\in W$, the module $\mathrm{C}_w P$ is a (generalized) tilting
module in $\mathcal{O}_0$ because $\mathrm{C}_w$ is a derived 
self-equivalence.

Thus, a problem is to determine for which
$w$ the category $\mathcal{O}_0$ is $\mathrm{C}_w P$-regular.
This problem looks much harder than the one involving the
twisting functors, due to the fact that shuffling functors
do not commute with projective functors.

\subsection{Regularity with respect to shuffled projectives}\label{s9.2}

\begin{proposition}\label{prop9.n1}
If $s$ is a simple reflection, then $\mathcal{O}_0$ is 
$\mathrm{C}_s P$-regular.
\end{proposition}

\begin{proof}
The functor $\mathrm{C}_s$ is defined as the cokernel of the
adjunction morphism $\mathrm{adj}_s:\theta_e\to\theta_s$.
If $x\in W$ is such that $xs<x$, then $\mathrm{C}_s P_x\cong P_x$.
If $x\in W$ is such that $xs>x$, then $\mathrm{C}_s P_x$
has projective dimension $1$ and a minimal projective resolution
of the following form:
\begin{equation}\label{eq9.n1-1}
0\to P_x\to\theta_s P_x\to  \mathrm{C}_s P_x \to 0,
\end{equation}
where any summand $P_y$ of $\theta_s P_x$ satisfies
$ys<y$ and hence $\mathrm{C}_s P_y=P_y$. 

The latter implies that \eqref{eq9.n1-1} can be viewed as a
coresolution of $P_x$ by modules in $\mathrm{add}(\mathrm{C}_s P)$
and it is manifestly regular. The claim follows.
\end{proof}

Proposition~\ref{prop9.n1} and Theorem \ref{thm8.n1} motivate the following:

\begin{conjecture}\label{conj5.n2}
If $w_0^{\mathfrak{p}}$ is the longest element in some 
parabolic subgroup of $W$, then 
$\mathcal{O}_0$ is  $\mathrm{C}_{w_0^{\mathfrak{p}}} P$-regular.
\end{conjecture}

Similarly to Subsection~\ref{s8.5} one can show that 
$\mathcal{O}_0$ is not $\mathrm{C}_{st} P$-regular
for $\mathfrak{g}=\mathfrak{sl}_3$.

\subsection{Shuffled tilting modules}\label{s9.3}

Let $T$ be a characteristic tilting module for $\mathcal{O}_0$.
For $w\in W$, the module $\mathrm{C}_w T$ is a (generalized) tilting
module in $\mathcal{O}_0$ because $\mathrm{C}_w$ induces a derived 
self-equivalence which is acyclic on tilting modules
(the latter follows by combining \cite[Proposition~5.3]{MS}
and \cite[Theorem~5.16]{MS}).

It seems to be an interesting problem to determine, for which
$w$, the category $\mathcal{O}_0$ is $\mathrm{C}_w T$-regular.
Again, this problem looks much harder than the one involving the
twisting functors due to the fact that shuffling functors
do not commute with projective functors.

\subsection{Regularity with respect to shuffled tiltings}\label{s9.4}

\begin{proposition}\label{prop9.n2}
If $s$ is a simple reflection, then $\mathcal{O}_0$ is 
$\mathrm{C}_s T$-regular.
\end{proposition}

\begin{proof}
This is very similar to the proof of Proposition~\ref{prop9.n1}.
If $x\in W$ is such that $xs>x$, then $\mathrm{C}_s T_x\cong T_x$.
If $x\in W$ is such that $xs<x$, then $\mathrm{C}_s T_x$
has a tilting resolution of the following form:
\begin{equation}\label{eq9.n1-2}
0\to T_x\to\theta_s T_x\to  \mathrm{C}_s T_x \to 0,
\end{equation}
where any summand $T_y$ of $\theta_s T_x$ satisfies
$ys>y$ and hence $\mathrm{C}_s T_y=T_y$.
Also, since $\theta_s$ is exact, the projective dimension of
$\theta_s T_x$ does not exceed that of $T_x$. Consequently, 
the projective dimension of $\mathrm{C}_s T_x$ is bounded by the
projective dimension of $T_x$ plus $1$.

We can now take a minimal tilting coresolution of $P$,
which we know has the regularity property, and coresolve each summand 
$T_x$, for $xs<x$, in this resolution using \eqref{eq9.n1-2}.
The outcome is a regular coresolution of $P$
by modules in $\mathrm{add}(\mathrm{C}_sT)$.
This completes the proof.
\end{proof}

Proposition~\ref{prop9.n2} motivates the following:

\begin{conjecture}\label{conj5.n2T}
If $w_0^{\mathfrak{p}}$ is the longest element in some 
parabolic subgroup of $W$, then 
$\mathcal{O}_0$ is  $\mathrm{C}_{w_0^{\mathfrak{p}}} T$-regular.
\end{conjecture}

Similarly to Subsection~\ref{s8.5} one can show that 
$\mathcal{O}_0$ is not $\mathrm{C}_{st} T$-regular
for $\mathfrak{g}=\mathfrak{sl}_3$.

\subsection{$\mathfrak{sl}_3$-example}\label{s9.5}

Let $\mathfrak g=\mathfrak{sl}_3$. Denote $W=\{e,s,t,st,ts,w_0=sts=tst\}$ as before. 

The left of the two tables below describes the projective dimensions 
of the twisted projective modules $\mathrm{C}_x P_y$. The right table
below describes the projective dimensions 
of the twisted tilting modules $\mathrm{C}_x T_y$. 
\begin{displaymath}
\begin{array}{c||c|c|c|c|c|c}
x\backslash y&e&s&t&st&ts&w_0\\ \hline\hline
e&0&0&0&0&0&0\\\hline
s&1&0&1&1&0&0\\\hline
t&1&1&0&0&1&0\\\hline
st&2&1&1&1&1&0\\\hline
ts&2&1&1&1&1&0\\\hline
w_0&3&1&1&1&1&0
\end{array}\qquad\qquad
\begin{array}{c||c|c|c|c|c|c}
x\backslash y&e&s&t&st&ts&w_0\\ \hline\hline
e&0&1&1&1&1&3\\\hline
s&0&2&1&1&2&4\\\hline
t&0&1&2&2&1&4\\\hline
st&0&2&2&2&2&5\\\hline
ts&0&2&2&2&2&5\\\hline
w_0&0&2&2&2&2&6
\end{array} 
\end{displaymath}

In the examples below,  we note the following difference with the case of
twisting functors: we do  not know whether the notion of
a ``tilting  core'' makes sense for shuffled projective and tilting modules.
Here are the graded characters of the modules $\mathrm{C}_s P_x$:

\resizebox{\textwidth}{!}{
$$
\xymatrix@C=0.1em@R=0.1em{
\mathrm{deg}\backslash {\color{blue}x}&&&
{\color{blue}e}&&&&{\color{blue}s}&&&&{\color{blue}t}&&&
{\color{blue}st}&&&&&&{\color{blue}ts}&&&&&{\color{blue}w_0}&&&\\
-1&|&&&&|&&s&&|&&&|&&&&&|&&&ts&&&|&&&{{}w_0}&&&\\
0&|&&s&&|&st&e&ts&|&ts&&|&&{{}w_0}&s&&|&&t&{{}w_0}&s&&|&{{}st}&&{{}ts}&&\\
1&|&st&&ts&|&s&{{}w_0}&t&|&s&{{}w_0}&|&{{}ts}&e&{{}st}&ts&|&ts&{{}st}&e&{{}ts}&st&|&{{}w_0}&{{}s}&&{{}t}&{{}w_0}&\\
2&|&&{{}w_0}&&|&&{{}st}&ts&|&ts&{{}st}&|&{{}w_0}&{{}s}&{{}w_0}&t&|&&{{}w_0}&{{}t}&{{}w_0}&s&|&{{}st}&{{}ts}&{{}e}&{{}st}&{{}ts}&\\
3&|&&&&|&&{{}w_0}&&|&{{}w_0}&&|&{{}ts}&&{{}st}&&|&&{{}st}&&{{}ts}&&|&{{}w_0}&{{}s}&&{{}t}&{{}w_0}&\\
4&|&&&&|&&&&|&&&|&&{{}w_0}&&&|&&&{{}w_0}&&&|&&{{}st}&&{{}ts}&&\\
5&|&&&&|&&&&|&&&|&&&&&|&&&&&&|&&&{{}w_0}&&&
}
$$
}

Here are the graded characters of the modules $\mathrm{C}_{st} P_x$:

\resizebox{\textwidth}{!}{
$$
\xymatrix@C=0.1em@R=0.1em{
\mathrm{deg}\backslash {\color{blue}x}&&
{\color{blue}e}&&&{\color{blue}s}&&{\color{blue}t}&&&
{\color{blue}st}&&&&{\color{blue}ts}&&&&&&{\color{blue}w_0}&&\\
-2&|&&|&&&|&&|&&&&|&&&&&|&&&{{}w_0}&&&\\
-1&|&&|&st&&|&&|&&{{}w_0}&&|&&{{}w_0}&t&&|&&{{}st}&&{{}ts}&&\\
0&|&st&|&t&{{}w_0}&|&{{}w_0}&|&{{}st}&&ts&|&{{}st}&e&{{}ts}&st&|&{{}w_0}&{{}s}&&{{}t}&{{}w_0}&\\
1&|&{{}w_0}&|&st&{{}ts}&|&{{}st}&|&{{}w_0}&t&{{}w_0}&|&{{}w_0}&{{}s}&{{}w_0}&t&|&{{}st}&{{}ts}&{{}e}&{{}st}&{{}ts}&\\
2&|&&|&&{{}w_0}&|&{{}w_0}&|&{{}st}&&{{}ts}&|&{{}st}&&{{}ts}&&|&{{}w_0}&{{}s}&&{{}t}&{{}w_0}&\\
3&|&&|&&&|&&|&&{{}w_0}&&|&&{{}w_0}&&&|&&{{}st}&&{{}ts}&&\\
4&|&&|&&&|&&|&&&&|&&&&&|&&&{{}w_0}&&&
}
$$
}

Here are the graded characters of the modules $\mathrm{C}_s T_x$:

\resizebox{\textwidth}{!}{
$$
\xymatrix@C=0.1em@R=0.1em{
\mathrm{deg}\backslash {\color{blue}x}&&
{\color{blue}w_0}&&{\color{blue}st}&&{\color{blue}ts}&&&&
{\color{blue}s}&&&&&{\color{blue}t}&&&&&{\color{blue}e}&&&\\
-4&|&&|&&|&&&|&&&&&|&&&&|&&&{{}w_0}&&&\\
-3&|&&|&&|&&&|&&{{}w_0}&&&|&&{{}w_0}&&|&&{{}st}&&{{}ts}&&\\
-2&|&&|&{{}w_0}&|&{{}w_0}&&|&{{}st}&&{{}ts}&&|&{{}st}&&{{}ts}&|&{{}w_0}&{{}s}&&{{}t}&{{}w_0}&\\
-1&|&{{}w_0}&|&{{}st}&|&{{}ts}&st&|&{{}w_0}&{{}s}&{{}w_0}&t&|&{{}w_0}&{{}t}&{{}w_0}&|&{{}st}&{{}ts}&{{}e}&{{}st}&{{}st}&\\
0&|&ts&|&{{}w_0}&|&{{}w_0}&t&|&{{}st}&ts&{{}st}&e&|&{{}st}&&{{}ts}&|&{{}w_0}&{{}s}&&{{}t}&{{}w_0}&\\
1&|&&|&&|&&st&|&&{{}w_0}&t&&|&&{{}w_0}&&|&&{{}st}&&{{}ts}&&\\
2&|&&|&&|&&&|&&&&&|&&&&|&&&{{}w_0}&&&
}
$$
}

Here are the graded characters of the modules  $\mathrm{C}_{st} T_x$:

\resizebox{\textwidth}{!}{
$$
\xymatrix@C=0.1em@R=0.1em{
\mathrm{deg}\backslash {\color{blue}x}&&&
{\color{blue}w_0}&&&{\color{blue}st}&&&&{\color{blue}ts}&&&&
{\color{blue}s}&&&&&&{\color{blue}t}&&&&&&{\color{blue}e}&&&\\
-5&|&&&&|&&&|&&&&|&&&&&&|&&&&&|&&&{{}w_0}&&&\\
-4&|&&&&|&&&|&&&&|&&{{}w_0}&&&&|&&{{}w_0}&&&|&&{{}st}&&{{}ts}&&\\
-3&|&&&&|&{{}w_0}&&|&&{{}w_0}&&|&{{}st}&&{{}ts}&&&|&{{}st}&&{{}st}&&|&{{}w_0}&{{}s}&&{{}t}&{{}w_0}&\\
-2&|&&{{}w_0}&&|&{{}ts}&st&|&st&{{}ts}&&|&{{}w_0}&{{}s}&{{}w_0}&t&&|&{{}w_0}&{{}t}&{{}w_0}&s&|&{{}st}&{{}ts}&{{}e}&{{}st}&{{}ts}&\\
-1&|&st&&ts&|&{{}w_0}&s&|&s&{{}w_0}&t&|&{{}st}&ts&{{}ts}&e&st&|&{{}ts}&e&{{}ts}&st&|&{{}w_0}&{{}s}&&{{}t}&{{}w_0}&\\
0&|&&s&&|&&st&|&st&e&ts&|&&{{}w_0}&t&s&&|&&{{}w_0}&s&&|&&{{}st}&&{{}ts}&&\\
1&|&&&&|&&&|&&s&&|&&&ts&&&|&&&&&|&&&{{}w_0}&&&
}
$$
}

The non-trivial  (ungraded) coresolutions of projectives using 
$\mathrm{C}_s P$ are:
\begin{gather*}
0\to{\color{blue}P_e}\to \mathrm{C}_s P_s\to \mathrm{C}_s P_e\to 0,\\
0\to{\color{blue}P_t}\to \mathrm{C}_s P_{ts}\to \mathrm{C}_s P_t\to 0,\\
0\to{\color{blue}P_{st}}\to \mathrm{C}_s P_s\oplus 
\mathrm{C}_s P_s\to \mathrm{C}_s P_{st}\to 0.
\end{gather*}
These all are, clearly, regular.

Next we claim that $P_e$ does not have a regular coresolution using 
$\mathrm{C}_{st} P$. Indeed, to have a chance at the zero step, we must 
embed $P_e$ into $\mathrm{C}_{st} P_{w_0}$. Let $\mathrm{Coker}$  be the cokernel.
In order to embed $\mathrm{Coker}$, in the next step we need a copy of 
$\mathrm{C}_{st} P_{st}$ or $\mathrm{C}_{st} P_{w_0}$ and another copy of 
$\mathrm{C}_{st} P_{ts}$ or $\mathrm{C}_{st} P_{w_0}$. Either way, the 
new cokernel will have a copy of $L_t$ in the socle, while it is easy to see that
no module in $\mathrm{add}(\mathrm{C}_{st} P)$ has $L_t$ in the socle, a contradiction.

The non-trivial  (ungraded) coresolutions of projectives using 
$\mathrm{C}_s T$ are:
\begin{gather*}
\resizebox{\textwidth}{!}{$
0\to{\color{blue}P_e}\to \mathrm{C}_s T_e\to \mathrm{C}_s T_t
\oplus \mathrm{C}_s T_e\oplus \mathrm{C}_s T_{st}\to
\mathrm{C}_s T_s\oplus \mathrm{C}_s T_{st}\oplus \mathrm{C}_s T_{t}\to
\mathrm{C}_s T_{ts}\oplus \mathrm{C}_s T_{st}\to \mathrm{C}_s T_{w_0}
\to 0,$}\\
0\to{\color{blue}P_s}\to \mathrm{C}_s T_{e}\to \mathrm{C}_s T_t\to 0,\\
0\to{\color{blue}P_t}\to \mathrm{C}_s T_{e}\to \mathrm{C}_s T_e\oplus
\mathrm{C}_s T_{st}\to \mathrm{C}_s T_{s}\to 0,\\
0\to{\color{blue}P_{st}}\to \mathrm{C}_s T_e\to \mathrm{C}_s T_{t}\to\mathrm{C}_s T_{ts}\to 0\\
0\to{\color{blue}P_{ts}}\to \mathrm{C}_s T_e\to \mathrm{C}_s T_{st}\to 0.
\end{gather*}
These all are, clearly, regular.

\section{Projective dimension of indecomposable twisted and shuffled projectives and tiltings}\label{s10}

\subsection{Projective dimension of twisted projectives}\label{s10.1}

The results of Subsection~\ref{s8.2} motivate the problem to determine
the projective dimension of twisted projective modules in $\mathcal{O}$.
Since twisting functors commute with projective functors, twisted
projective modules are exactly the modules obtained by applying
projective functors to Verma modules:
\begin{equation}
    \top_xP_y\cong \top_x\theta_y\Delta_e\cong \theta_y\top_x\Delta_e\cong \theta_y\Delta_x.
\end{equation}
This allows us to reformulate
the problem as follows:

\begin{problem}\label{probs10-1}
For $x,y\in W$, determine the projective dimension of the
module $\theta_x\Delta_y$.
\end{problem}

Here are some basic observations about this problem:
\begin{itemize}
\item If $y=e$, the module $\theta_x\Delta_e$ is projective
and hence the answer is $0$.
\item If $y=w_0$, the module $\theta_x\Delta_{w_0}$ is 
a tilting module and hence the answer is $\mathbf{a}(w_0x)$,
see \cite{Ma3,Ma4}.
\item If $x=e$, the answer is $\ell(y)$, see \cite{Ma3}.
\item If $x=w_0$, we have $\theta_{w_0}\Delta_y=P_{w_0}$
and the answer is $0$.
\item For a fixed $y$, the answer is weakly monotone in $x$, with respect to 
the right Kazhdan-Lusztig order, in particular,
the answer is constant on the right Kazhdan-Lusztig cell
of $x$. 
\item For a simple reflection $s$, we have 
$\theta_x\Delta_y=\theta_x\Delta_{ys}$ provided that
$\ell(sx)<\ell(x)$, in particular, it is enough to consider the 
situation where $x$ is a Duflo involution and $y$ is a shortest
(or longest) element in a coset from $W/W'$, where $W'$ is the 
parabolic subgroup of $W$ generated by all simple reflections
in the left descent set of $x$.
\item If $x=w_0^{\mathfrak{p}}$, for some parabolic 
$\mathfrak{p}$, then the projective dimension of 
$\theta_{w_0^{\mathfrak{p}}}\Delta_y$ coincides with the
projective dimension of the singular Verma module obtained
by translating $\Delta_y$ to the wall corresponding to 
$w_0^{\mathfrak{p}}$. This can be computed in therms of a 
certain function $\mathtt{d}_\lambda$, see \cite[Table~2]{CM}
(see also \cite[Formula~(1.2)]{CM} and \cite[Remark~6.9]{KMM}).
\end{itemize}

The last observation suggest that Problem~\ref{probs10-1}
might be not easy. Also, note that, by Koszul duality, the problem to
determine the projective dimension of a singular Verma module
is equivalent to the problem to determine the graded length of
a parabolic Verma module. The latter is certainly ``combinatorial''
in the sense that the answer can be formulated purely in terms
of Kazhdan-Lusztig combinatorics.

Let $\mathbf{H}$ denote the Hecke algebra 
of $W$ (over $\mathbb{A}=\mathbb{Z}[v,v^{-1}]$ and in the normalization
of \cite{So3}) with standard basis
$\{H_w\,:\,w\in W\}$ and Kazhdan-Lusztig basis 
$\{\underline{H}_w\,:\,w\in W\}$. Consider the structure constants
$h_{x,y}^z\in \mathbb{A}$ with respect to the KL-basis, that is
\begin{displaymath}
\underline{H}_x\underline{H}_y=\sum_{z\in W}
h_{x,y}^z\underline{H}_z.
\end{displaymath}
In \cite[Subsection~6.3]{KMM}, for $x,y\in W$, 
we defined the function 
$\mathbf{b}:W\times W\to \mathbb{Z}_{\geq 0} \sqcup \{-\infty\}$ as follows:
\begin{displaymath}
\mathbf{b}(x,y):=\max\{\deg(h_{z,x^{-1}}^y)\,:\,z\in W\}.
\end{displaymath}

(By our convention the degree of the zero polynomial is $-\infty$.)
The value $\mathbf{b}(x,y)$ is, if not $-\infty$, equal to the maximal degree of a non-zero graded component of $\theta_x L_y$, and also to the maximal non-zero position in the minimal complex of tilting modules representing $\theta_{y^{-1}w_0} L_{w_0x^{-1}}$.

Here is an upper bound for the projective dimension of 
$\theta_x\Delta_y$ expressed in terms of the $\mathbf{b}$-function.

\begin{proposition}\label{props10-n5}
For all $x,y\in W$, we have:
\begin{enumerate}
\item\label{props10-n5.1} $\mathrm{proj.dim}\,\theta_x\Delta_y\leq 
\max\{\mathbf{b}(w_0a^{-1}w_0,x^{-1}w_0)\,:\,a\leq y\}$.
\item\label{props10-n5.2} 
If the maximum in \eqref{props10-n5.1} coincides with
$\mathbf{b}(w_0y^{-1}w_0,x^{-1}w_0)$, then the latter value
is equal to  $\mathrm{proj.dim}\,\theta_x\Delta_y$.
\end{enumerate}
\end{proposition}

\begin{proof}
For $x,y,z\in W$ and $k\in\mathbb{Z}_{\geq 0}$, 
by adjunction, we have
\begin{displaymath}
\mathrm{Ext}^k_{\mathcal{O}}(\theta_x\Delta_y,L_z)\cong
\mathrm{Ext}^k_{\mathcal{O}}(\Delta_y,\theta_{x^{-1}}L_z).
\end{displaymath}
By \cite{Ma4}, the module $\theta_{x^{-1}}L_z$ can be represented 
by a linear complex of tilting module. Moreover, the multiplicity
of $T_{a}\langle k\rangle[-k]$ in this complex coincides with the
composition multiplicity of $L_{w_0a^{-1}w_0}\langle k\rangle$ 
in $\theta_{z^{-1}w_0}L_{w_0x}$.

A costandard filtration of $T_{a}\langle k\rangle[-k]$ can contain $\nabla_y$ only when
$a\leq y$, and hence only
such summand $T_{a}\langle k\rangle[-k]$ in the tilting complex can, potentially, give rise to a non-zero element
in $\mathrm{Ext}^k_{\mathcal{O}}(\Delta_y,\theta_{x^{-1}}L_z)$.
Here we use the fact that standard and costandard modules
are homologically orthogonal and hence derived homomorphisms
can be constructed already on the level of the homotopy category.
This implies claim~\eqref{props10-n5.1}.

To prove claim~\eqref{props10-n5.2}, assume
\begin{displaymath}
k:= \mathbf{b}(w_0y^{-1}w_0,x^{-1}w_0)=
\max\{\mathbf{b}(w_0a^{-1}w_0,x^{-1}w_0)\,:\,a\leq y\}.
\end{displaymath}
The canonical map $\Delta_y\to T_y$ gives rise to a homomorphism
of $\Delta_y\langle k\rangle$ to the $k$-th homological 
position of the linear complex of tilting modules
representing $\theta_{x^{-1}}L_z$. Because of the maximality
assumption on $k$, there are no homomorphisms from $\Delta_y$
to the $k+1$-st homological position. This means that 
the map from the previous sentence is a homomorphism of complexes.
It is not homotopic to zero since since the complex representing
$\theta_{x^{-1}}L_z$ is linear and $T_{y}\langle k\rangle[-k]$
is in a diagonal position in this complex. The corresponding level
at the position $k-1$ does not contain any socles of any 
costandard modules since all indecomposable tilting summands there
are shifted by one in the positive direction of the grading.
This means that the map we constructed gives a non-zero 
extension. Hence claim~\eqref{props10-n5.2} now follows
from claim~\eqref{props10-n5.1}.
\end{proof}

\begin{corollary}\label{cor10-n6}
For any parabolic $\mathfrak{p}$, 
in case $x\leq_{\mathtt{R}}w^{\mathfrak{p}}_0w_0$, we have
$\mathrm{proj.dim}\,\theta_x\Delta_{w^{\mathfrak{p}}_0}=
\ell(w^{\mathfrak{p}}_0)$.
\end{corollary}

\begin{proof}
If $x\leq_{\mathtt{R}}w^{\mathfrak{p}}_0w_0$,
then \cite[Proposition~6.8]{KMM} implies
$\mathbf{b}(w_0w^{\mathfrak{p}}_0 w_0,x^{-1}w_0)=\ell(w^{\mathfrak{p}}_0)$.
For any $a\leq \wi$, we also have
\begin{displaymath}
\mathbf{b}(w_0aw_0,x^{-1}w_0)\leq \ell(a)\leq
\ell(w^{\mathfrak{p}}_0)=\mathbf{b}(w_0w^{\mathfrak{p}}_0 w_0,x^{-1}w_0),
\end{displaymath}
also using \cite[Proposition~6.8]{KMM}.
Hence the claim follows from
Proposition~\ref{props10-n5}\eqref{props10-n5.2}.
\end{proof}

\subsection{Projective dimension of twisted tiltings}\label{s10.3}

The results of Subsection~\ref{s8.4} motivate the problem to determine
the projective dimension of twisted tilting modules in $\mathcal{O}$.
By
\begin{equation}
    \top_xT_{w_0y}\cong \top_x\theta_yT_{w_0}\cong\top_x\theta_y\nabla_{w_0}\cong \theta_y\top_x\nabla_{w_0}\cong \theta_y\nabla_{xw_0},
\end{equation}
we reformulate
the problem as follows:

\begin{problem}\label{probs10-15}
For $x,y\in W$, determine the projective dimension of the
module $\theta_x\nabla_y$.
\end{problem}

Here are some basic observations about this problem:
\begin{itemize}
\item If $y=w_0$, the module $\theta_x\nabla_{w_0}$ is tilting
and hence the answer is $\mathbf{a}(w_0x)$, see  \cite{Ma3,Ma4}.
\item If $y=e$, the module $\theta_x\nabla_{e}$ is 
an indecomposable injective module and hence the answer is 
$2\mathbf{a}(w_0x)$, see \cite{Ma3,Ma4}.
\item If $x=e$, the answer is $2\ell(w_0)-\ell(y)$, see \cite{Ma3}.
\item If $x=w_0$, we have $\theta_{w_0}\nabla_y=P_{w_0}$
and the answer is $0$.
\item For a fixed $y$, the answer is weakly monotone in $x$, with respect to 
the right Kazhdan-Lusztig order, in particular,
the answer is constant on the right Kazhdan-Lusztig cell
of $x$. 
\item For a simple reflection $s$, we have 
$\theta_x\nabla_y=\theta_x\nabla_{ys}$ provided that
$\ell(sx)<\ell(x)$, in particular, it is enough to consider the 
situation where $x$ is a Duflo involution and $y$ is a shortest
(or longest) element in a coset from $W/W'$, where $W'$ is the 
parabolic subgroup of $W$ generated by all simple reflections
in the left descent set of $x$.
\item If $x=w_0^{\mathfrak{p}}$, for some parabolic 
$\mathfrak{p}$, then the projective dimension of 
$\theta_{w_0^{\mathfrak{p}}}\nabla_y$ coincides with the
projective dimension of the singular dual Verma module obtained
by translating $\nabla_y$ to the wall corresponding to 
$w_0^{\mathfrak{p}}$. This can be computed in therms of a 
certain function $\mathtt{d}_\lambda$, see \cite[Table~2]{CM}
(see also \cite[Formula~(1.2)]{CM} and \cite[Remark~6.9]{KMM}).
\end{itemize}

Let us now  observe that $\nabla_y\cong \top_{w_0}\Delta_{w_0y}$
and that $\theta_x\nabla_y\cong \top_{w_0}\theta_x\Delta_{w_0y}$
since twisting and projective functors commute. We conjecture
the following connection between Problems~\ref{probs10-1}
and \ref{probs10-15}.

\begin{conjecture}\label{conj10-151}
For $x,y\in W$, we have
$\mathrm{proj.dim}\, \theta_x\nabla_y=\mathbf{a}(w_0 x)+
\mathrm{proj.dim}\,\theta_x\Delta_{w_0y}$.
\end{conjecture}

Below we present some evidence for 
Conjecture~\ref{conj10-151}.

\begin{proposition}\label{prop10-152}
For $x,y\in W$, we have
$\mathrm{proj.dim}\, \theta_x\nabla_y\leq \mathbf{a}(w_0 x)+
\mathrm{proj.dim}\,\theta_x\Delta_{w_0y}$.
\end{proposition}

\begin{proof}
Assume that $\mathrm{proj.dim}\,\theta_x\Delta_{w_0y}=k$ and
let $\mathcal{P}_{\bullet}$ be a minimal projective resolution 
of $\theta_x\Delta_{w_0y}$. Applying $\top_{w_0}$ to 
$\mathcal{P}_{\bullet}$, we get a minimal tilting resolution 
of $\theta_x\nabla_y$ (of length $k$). To obtain a projective
resolution of $\theta_x\nabla_y$, we need to projectively resolve
each indecomposable tilting module $T_u$ appearing in
$\top_{w_0}\mathcal{P}_{\bullet}$ and glue all these resolutions
together. In particular, $\mathrm{proj.dim}\, \theta_x\nabla_y$ 
is bounded by $k$ plus the maximal value of $\mathrm{proj.dim}\, T_u$,
for $T_u$ appearing in $\top_{w_0}\mathcal{P}_{\bullet}$.

Note that any indecomposable projective $P_v$ appearing in
$\mathcal{P}_{\bullet}$ satisfies $v\geq_\mathtt{L} x$, because it is 
a summand of $\theta_x P_w$, for some $w$. Therefore
$u=w_0v$ satisfies $u\leq_\mathtt{L} w_0x$. In particular, we have
$\mathbf{a}(u)\leq \mathbf{a}(w_0x)$. By \cite{Ma3,Ma4},
the projective dimension of $T_u$ equals $\mathbf{a}(u)$.
The claim of  the proposition follows.
\end{proof}

\begin{corollary}\label{cor10-153}
For $x,y\in W$, let $\mathrm{proj.dim}\,\theta_x\Delta_{w_0y}=k$.
Assume that there exists $v\in W$ such that
$v\sim_\mathtt{L} x$ and 
$\mathrm{Ext}^k_{\mathcal{O}}(\theta_x\Delta_{w_0y},L_v)\neq 0$.
Then $\mathrm{proj.dim}\, \theta_x\nabla_y= \mathbf{a}(w_0 x)+
\mathrm{proj.dim}\,\theta_x\Delta_{w_0y}$.
\end{corollary}

\begin{proof}
Let us look closely at the proof of Proposition~\ref{prop10-152}.
From \cite[Section~6]{KMM}, it follows that there exists $w\in W$ such that $T_{w_0v}$
appears in position $\mathbf{a}(w_0x)$ of a minimal tilting complex
$\mathcal{T}_\bullet$ representing $L_w$ and, moreover, this 
position $\mathbf{a}(w_0x)$ is a maximal non-zero position in
$\mathcal{T}_\bullet$.

The module $T_{w_0v}$ appears as a summand in 
$\top_{w_0}\mathcal{P}_{-k}$ and in  $\mathcal{T}_{\mathbf{a}(w_0x)}$.
Similarly to the proof of \cite[Theorem~1]{MO2}, the identity map on $T_{w_0v}$
induces a non-zero map from $\top_{w_0}\mathcal{P}_{\bullet}$
to $\mathcal{T}_\bullet[\mathbf{a}(w_0 x)+k]$
in the homotopy category and hence gives rise
to a non-zero extension fro $\theta_x\nabla_y$ to
$L_w$ of degree $\mathbf{a}(w_0 x)+k$, by construction.
Therefore 
$\mathrm{proj.dim}\, \theta_x\nabla_y\geq \mathbf{a}(w_0 x)+
\mathrm{proj.dim}\,\theta_x\Delta_{w_0y}$
and the claim of the corollary follows from 
Proposition~\ref{prop10-152}.
\end{proof}

We note that the condition ``there exists $v\in W$ such 
that $v\sim_\mathtt{L} x$ and  
$\mathrm{Ext}^k_{\mathcal{O}}(\theta_x\Delta_{w_0y},L_v)\neq 0$''
in Corollary~\ref{cor10-153}
is very similar to  \cite[Conjecture~1.3]{KMM} proved
in \cite[Theorem~A]{KMM}. We suspect that this condition is always satisfied.

\subsection{Projective dimension of shuffled projectives}\label{s10.2}

The results of Subsection~\ref{s9.2} motivate the problem to determine
the projective dimension of shuffled projective modules in $\mathcal{O}$.

\begin{problem}\label{probs10-2}
For $x,y\in W$, determine the projective dimension of the
module $\mathrm{C}_x P_y$.
\end{problem}

This problems looks much harder than the one for the 
twisted projective modules, mostly because twisting functors
do not commute with projective functors, in general.

Here are some basic observations about this problem:
\begin{itemize}
\item If $x=e$, the module $\mathrm{C}_e P_y$ is projective
and hence the answer is $0$.
\item If $x=w_0$, the module $\mathrm{C}_{w_0} P_y$ is 
the tilting module $T_{yw_0}$ (this follows from \cite[Proposition~2.2, Proposition~4.4]{MS2} by a character argument). Hence the answer is 
$\mathbf{a}(yw_0)$ by \cite{Ma3,Ma4}.
\item If $y=e$, we have $\mathrm{C}_x P_e\cong\mathrm C_x \Delta_e\cong \Delta_x$ and
the answer is $\ell(x)$, see \cite{MS,Ma3}.
\item If $y=w_0$, we have $\mathrm{C}_x P_{w_0}=P_{w_0}$
and the answer is $0$. 
\item The projective dimension of  $\mathrm{C}_x P_y$
is at most $\ell(x)$, since each $\mathrm{C}_s$, where
$s$ is a simple reflection, has derived length $1$.
\item For $x=s$, a simple reflection, we have
$\mathrm{C}_s P_y\cong P_y$ if $ys<y$, in which case the answer
is $0$. In case $ys>y$, the module $\mathrm{C}_s P_y$ is
not projective and the answer is $1$, see the proof
of Proposition~\ref{prop9.n1}.
\end{itemize}

In the spirit of Subsection~\ref{s7.3}, we can reformulate
Problem~\ref{probs10-2} in terms of the cohomology of certain 
functors. For $w\in W$, we denote by $\mathrm{K}_w$ the right 
adjoint of $\mathrm{C}_w$, called the {\em coshuffling} functor, 
see \cite[Section~5]{MS}. Note that, for a reduced expression
$w=rs\dots t$, we have $\mathrm{C}_w=\mathrm{C}_t\dots
\mathrm{C}_s\mathrm{C}_r$ and $\mathrm{K}_w=\mathrm{K}_r
\mathrm{K}_s\dots\mathrm{K}_t$. Also, we have
$\mathrm{K}_w=\star\circ\mathrm{C}_w\circ\star$.
We denote by $L$ the direct sum of all simple
modules in $\mathcal{O}_0$.

\begin{proposition}\label{props10-21}
For $x,y\in W$, the projective dimension of $\mathrm{C}_x P_y$
coincides with the maximal $k\geq 0$ such that 
$[\mathcal{L}_k\mathrm{C}_x\, L:L_y]\neq 0$.
\end{proposition}

\begin{proof}
The projective dimension of a module coincides with the maximal
degree of a non-vanishing extension to a simple module.
Since $\mathcal{L}\mathrm{C}_x$ is a derived equivalence with
inverse $\mathcal{R}\mathrm{K}_x$ by \cite[Theorem~5.7]{MS},
for $i\geq 0$, we have
\begin{displaymath}
\begin{array}{rcl} 
\dim \mathrm{Ext}_{\mathcal{O}}^i(\mathrm{C}_x P_y,L)&=&
\dim \mathrm{Hom}_{\mathcal{D}^b(\mathcal{O})}(\mathrm{C}_x P_y,L[i])\\
&=&\dim \mathrm{Hom}_{\mathcal{D}^b(\mathcal{O})}
(\mathcal{L}\mathrm{C}_x P_y,L[i])\\
&=&\dim \mathrm{Hom}_{\mathcal{D}^b(\mathcal{O})}
(P_y,\mathcal{R}\mathrm{K}_x L[i])\\
&=& [\mathcal{R}^i\mathrm{K}_x\, L:L_y]\\
&=& [\mathcal{L}_i\mathrm{C}_x\, L^\star:L_y^\star]\\
&=& [\mathcal{L}_i\mathrm{C}_x\, L:L_y]
\end{array}
\end{displaymath}
and the claim follows.
\end{proof}

\subsection{Projective dimension of shuffled tiltings}\label{s10.4}

The results of Subsection~\ref{s9.4} motivate the problem to determine
the projective dimension of shuffled projective modules in $\mathcal{O}$.

\begin{problem}\label{probs10-45}
For $x,y\in W$, determine the projective dimension of the
module $\mathrm{C}_x T_y$.
\end{problem}

This problems looks much harder than the one for the 
twisted tilting modules, mostly because twisting functors
do not commute with projective functors, in general.

Here are some basic observations about this problem:
\begin{itemize}
\item If $x=e$, the module $\mathrm{C}_e T_y$ is tilting
and hence the answer is $\mathbf{a}(y)$, see \cite{Ma3,Ma4}.
\item If $x=w_0$, the module $\mathrm{C}_{w_0} T_y$ is 
the injective module $I_{yw_0}$. In fact, we have
\[\mathrm{C}_{w_0} T_y \cong\mathrm{C}_{w_0} \top_{w_0} P_{w_0y}\cong \top_{w_0}\mathrm{C}_{w_0} P_{w_0y} \cong \top_{w_0}T_{w_0yw_0}\cong I_{yw_0}.\]
Hence the answer is 
$2\mathbf{a}(w_0yw_0)$ by \cite{Ma3,Ma4}.
\item If $y=e$, we have $\mathrm{C}_x T_e\cong \mathrm{C}_{x}\top_{w_0} P_{w_0} \cong \top_{w_0}\mathrm{C}_{x} P_{w_0} \cong \top_{w_0}P_{w_0} 
\cong P_{w_0}$ 
and the answer is $0$.
\item If $y=w_0$, we have $\mathrm{C}_x T_{w_0}\cong \mathrm{C}_{x} \nabla_{w_0} \cong\nabla_{w_0x}$
and the answer is $\ell(w_0)+\ell(x)$, see \cite{Ma3}. 
\item The projective dimension of  $\mathrm{C}_x T_y$
is at most $\ell(x)+\mathbf{a}(y)$, since the projective
dimension of $T_y$ is $\mathbf{a}(y)$ by \cite{Ma3,Ma4}
and each $\mathrm{C}_s$, where
$s$ is a simple reflection, has derived length $1$.
\item For $x=s$, a simple reflection, we have
$\mathrm{C}_s T_y\cong T_y$ if $ys>y$, in which case the answer
is $\mathbf{a}(y)$ by \cite{Ma3,Ma4}. 
In case $ys<y$, the module $\mathrm{C}_s P_y$ is
no longer tilting and the answer is $\mathbf{a}(y)+1$
because the minimal tilting resolution of $\mathrm{C}_s P_y$
has $T_y$ in position $-1$.
\end{itemize}

In the spirit of Subsection~\ref{s10.3}, we make the
following conjecture:

\begin{conjecture}\label{conj10-451}
For $x,y\in W$, we have
$\mathrm{proj.dim}\,\mathrm{C}_x T_y=\mathbf{a}(y)+
\mathrm{proj.dim}\,\mathrm{C}_x P_{w_0y}$.
\end{conjecture}

Below we present some evidence for 
Conjecture~\ref{conj10-451}.

\begin{proposition}\label{prop10-452}
For $x,y\in W$, we have
$\mathrm{proj.dim}\,\mathrm{C}_x T_y\leq \mathbf{a}(y)+
\mathrm{proj.dim}\,\mathrm{C}_x P_{w_0y}$.
\end{proposition}

\begin{proof}
Assume that $\mathrm{proj.dim}\,\mathrm{C}_x P_{w_0y}=k$ and
let $\mathcal{P}_{\bullet}$ be a minimal projective resolution 
of $\mathrm{C}_x P_{w_0y}$. Applying $\top_{w_0}$ to 
$\mathcal{P}_{\bullet}$, and using that twisting and shuffling
functors commute (e.g. because twisting functors commute with
projective functors and natural transformations between them and
shuffling functors are defined in terms of (co)kernels of such natural
transformations), we get a minimal tilting resolution 
of $\mathrm{C}_x T_y$ (of length $k$). To obtain a projective
resolution of $\mathrm{C}_x T_y$, we need to projectively resolve
each indecomposable tilting module $T_u$ appearing in
$\top_{w_0}\mathcal{P}_{\bullet}$ and glue all these resolutions
together. In particular, $\mathrm{proj.dim}\, \mathrm{C}_xT_y$ 
is bounded by $k$ plus the maximal value of $\mathrm{proj.dim}\, T_u$,
for $T_u$ appearing in $\top_{w_0}\mathcal{P}_{\bullet}$.

Note that a projective resolution of $\mathrm{C}_sP_{w}$, for any $w\in W$ and $s\in S$, has the following form: $0\to P_w\to\theta_x P_w\to0$ and a projective resolution of $\mathrm{C}_xP_{w_0y}$ is obtained by gluing such resolutions inductively along a reduced decomposition of $x$. Thus, an indecomposable projective $P_v$ appearing in
$\mathcal{P}_{\bullet}$  is 
a summand of $\theta P_{w_0y}$ for some projective functor $\theta$ and satisfies $v\geq_\mathtt{R} w_0y$. 
Therefore,
$u=w_0v$ satisfies $u\leq_\mathtt{R} y$. In particular, we have
$\mathbf{a}(u)\leq \mathbf{a}(y)$. By \cite{Ma3,Ma4},
the projective dimension of $T_u$ equals $\mathbf{a}(u)$.
The claim of  the proposition follows.
\end{proof}

\begin{corollary}\label{cor10-453}
For $x,y\in W$, let $\mathrm{proj.dim}\,\mathrm{C}_x P_{w_0y}=k$.
Assume that there exists $v\in W$ such that
$v\sim_\mathtt{R} w_0y$ and 
$\mathrm{Ext}^k_{\mathcal{O}}(\mathrm{C}_x P_{w_0y},L_v)\neq 0$.
Then $\mathrm{proj.dim}\, \mathrm{C}_x T_y=\mathbf{a}(y)+
\mathrm{proj.dim}\,\mathrm{C}_x P_{w_0y}$.
\end{corollary}

\begin{proof}

Follows from Proposition \ref{prop10-452} by 
a line of arguments analogous to the ones in the proof of Corollary \ref{cor10-153}.

%
%
\end{proof}

Again, we suspect that the above assumption ``there exists $v\in W$ such that
$v\sim_\mathtt{R} w_0y$ and 
$\mathrm{Ext}^k_{\mathcal{O}}(\mathrm{C}_x P_{w_0y},L_v)\neq 0$''
in Corollary~\ref{cor10-453}  is always
satisfied. 

\vspace{2mm}

\noindent
H.~K.: Department of Mathematics, Uppsala University, Box. 480,
SE-75106, Uppsala,\\ SWEDEN, email: {\tt hankyung.ko\symbol{64}math.uu.se}

\noindent
V.~M.: Department of Mathematics, Uppsala University, Box. 480,
SE-75106, Uppsala,\\ SWEDEN, email: {\tt mazor\symbol{64}math.uu.se}

\noindent
R.~M.: Department of Mathematics, Uppsala University, Box. 480,
SE-75106, Uppsala,\\ SWEDEN, email: {\tt rafael.mrden\symbol{64}math.uu.se}

\end{document}